\documentclass[11pt, fleqn, leqno, a4paper]{article}
\usepackage[utf8]{inputenc}
\usepackage[normalem]{ulem} 
\usepackage[T1]{fontenc}
\usepackage[english]{babel}
\usepackage{comment}
\usepackage{amsmath, amssymb, amsfonts, amsthm}
\newtheorem{assumption}{Assumption}
\usepackage{geometry}
\usepackage{graphicx}
\usepackage{parskip} 
\usepackage{hyperref}
\usepackage{mathtools} 
\usepackage{graphicx}  
\usepackage{geometry}  
\usepackage{multirow} 
\usepackage{authblk}  
\usepackage{hyperref} 
\usepackage{cite}
\newcommand{\wt}{\widetilde}
\newcommand{\wh}{\widehat}
\newcommand{\eps}{\varepsilon}
\def\ba{\mathbf{a}}
\def\bb{\mathbf{b}}

\def\bd{\mathbf{d}}
\def\be{\mathbf{e}}
\def\bg{\mathbf{g}}

\def\bn{\mathbf{n}}
\def\bp{\mathbf{p}}
\def\bq{\mathbf{q}}
\def\br{\mathbf{r}}
\def\bs{\mathbf{s}}
\def\bu{\mathbf{u}}
\def\bv{\mathbf{v}}
\def\bw{\mathbf{w}}
\def\bx{\mathbf{x}}
\def\bz{\mathbf{z}}
\def\by{\mathbf{y}}

\def\zero{\mathbf{0}}
\def\bxi{\boldsymbol{\xi}}
\newcommand{\ph}{\phantom}
\newcommand{\calo}{\mathcal{O}}

\newcommand{\R}{\mathbb{R}}

\newcommand{\gammac}{\overline{\gamma}}

\newcommand{\rhoc}{\overline{\rho}}
\newcommand{\mtxa}[2]{\left[\!\!
\begin{array}{#1} #2 \end{array} \!\! \right]}
\newcommand{\smtxa}[2]{
{\mbox{\scriptsize
$\left[\!\! \begin{array}{#1} #2 \end{array} \!\! \right]$}}}
\DeclareMathOperator*{\argmin}{arg\,min}
\usepackage{algorithm}
\usepackage{algorithmicx}
\usepackage{algpseudocode}
\usepackage{longtable}
\hypersetup{
    colorlinks=true,
    linkcolor=blue,
    urlcolor=cyan,
}
\usepackage{subcaption}
\usepackage{xcolor}

\definecolor{dorange}{RGB}{220,110,0}

\newtheorem{theorem}{Theorem}[section]
\newtheorem{lemma}[theorem]{Lemma}
\newtheorem{example}[theorem]{Example}
\newtheorem{proposition}[theorem]{Proposition}
\newtheorem{corollary}[theorem]{Corollary}
\theoremstyle{definition}

\theoremstyle{remark}
\newtheorem{remark}{Remark}[section]
\title{\textbf{ A Twin gradient method for unconstrained optimization}}

\author[1,2]{Anna De Magistris}
\author[3]{Michiel E.~Hochstenbach}
\author[1,2]{Gerardo Toraldo}

\affil[1]{\small Department of Mathematics and Physics, University of Campania ``Luigi Vanvitelli'', Viale Lincoln 5, 81100 Caserta, Italy}

\affil[2]{\small Member of the INdAM research group GNCS}

\affil[3]{\small Department of Mathematics and Computer Science, Eindhoven University of Technology, PO Box 513, 5600 MB Eindhoven, The Netherlands}

\date{}

\begin{document}

\maketitle

\begin{abstract}
We propose a new strategy for gradient-based unconstrained optimization, involving two parallel sequences of iterates that cooperate to determine their stepsizes via a \textit{Twin-Step} principle. Rather than minimizing the objective function individually, the algorithm selects steplengths that minimize the Euclidean distance between the two gradient based processes occurring simultaneously at each iteration. 
The theoretical analysis  shows that the convergence of the mutual distance is governed by the angle between the search directions. In particular the  effectiveness of the overall process degrades as the directions approach parallelism.
To ensure robustness against collinearity, we introduce a hybrid framework, Twin-ABB$_{\min}$, which switches to the Adaptive Barzilai--Borwein method when the geometric cooperation becomes ineffective.  Extensive and very promising numerical results evidence that the Twin phase creates favorable initial conditions for subsequent BB-type iterations.
\end{abstract}

\vspace{0.5cm}

\section{Introduction}
\label{sec:intro}
We develop a new gradient method for unconstrained optimization problems
\begin{equation}
\label{eq:general_problem}
\min_{\bx \in \mathbb{R}^n} f(\bx),
\end{equation}
where $f \in C^1$. Some theoretical results of the paper will need the assumption that $\nabla f$ is Lipschitz continuous with constant $L$.
We will also show some additional theoretical results for the strictly convex quadratic optimization problem
\begin{equation}
\label{eq:quadratic_problem}
\min_{\bx \in \mathbb{R}^n} f(\bx) = \tfrac12 \, \bx^{\top} \! A \bx - \bb^{\top} \bx,
\end{equation}
where $A \in \mathbb{R}^{n \times n}$ is symmetric positive definite with eigenvalues $0 < \lambda_1 \le \cdots \le \lambda_n$, and $\bb \in \mathbb{R}^n$.
This problem is of paramount importance in a wide variety of applications, ranging from signal and image processing to machine learning and compressed sensing (see, e.g., \cite{ loris2009accelerating,Serafini:2005, WrightNowakFigueiredo:2009SpaRSA, zanella2013towards, CRISCI2025116729}). Furthermore, the quadratic framework is an essential testing ground for algorithmic strategies that can be generalized to nonlinear, nonconvex, or large-scale optimization scenarios.

Gradient-based methods of the form $\bx_{k+1} = \bx_k - \alpha_k \, \nabla f(\bx_k)$
are popular due to their simplicity, low computational cost per iteration, and minimal storage requirements. Although the Steepest Descent method guarantees global convergence and monotonic decrease of the objective function for a quadratic function, it is known to have slow convergence rates for ill-conditioned problems, as demonstrated by its well-known ``zig--zag'' behavior.
To overcome the limitations of SD while ensuring the simplicity of gradient iterations, significant research has focused on the selection of the stepsize. A breakthrough in this area has been achieved by Barzilai and Borwein (BB) \cite{BarzilaiBorwein:1988}, who propose a spectral stepsize derived from a two-point approximation of the secant equation:
\[
\alpha_k^{\text{BB1}}=\frac{ \bs_{k-1}^\top\bs_{k-1}}{ \bs_{k-1}^{\top} \by_{k-1}}, \qquad \alpha_k^{\text{BB2}}=\frac{ \bs_{k-1}^{\top} \by_{k-1}}{ \by_{k-1}^\top\by_{k-1}},
\]
where $\bs_{k-1}= \bx_{k}- \bx_{k-1}, \; \by_{k-1}= \nabla f(\bx_{k})- \nabla f(\bx_{k-1})$. The BB methods generally offer far better performance compared to the SD method, despite not guaranteeing a monotonic decrease in the objective function. This success led to extensive research on spectral gradient methods, resulting in variants such as the Cyclic BB, Monotone Gradient methods, and Adaptive Barzilai--Borwein (ABB) strategies (see \cite{Dai:2002, Raydan:1997, Zhou:2006}).
In ABB we select, for a chosen $\tau \in (0,1)$,
\[
\alpha_k^{\rm{ABB}} = 
\begin{cases} 
\alpha_k^{\text{BB2}}  & \text{if } \frac{\alpha_k^{\text{BB2}}}{\alpha_k^{\text{BB1}}} < \tau, \\
\alpha_k^{\text{BB1}} & \text{otherwise}.
\end{cases}
\]
This is modified by Frassoldati et al.~\cite{Frassoldati:2008}, introducing the ABB$_{\rm {min}}$ strategy:
\[
\alpha_k^{\text{ABB}_{\rm{min}}} = 
\begin{cases} 
\min \{\alpha_j^{\text{BB2}} \mid j = \max(1, k - M_\alpha), \dots, k\} & \text{if } \frac{\alpha_k^{\text{BB2}}}{\alpha_k^{\text{BB1}}} < \tau, \; \tau \in (0,1) \\
\alpha_k^{\text{BB1}} & \text{otherwise}
\end{cases}
\]
where \( M_\alpha > 0 \) is a prefixed integer for memory usage.
These approaches aim to capture local curvature using information from previous iterations to accelerate convergence.
The development of efficient stepsizes for gradient methods remains an active area of research, driven by the limitations of traditional approaches. A significant and historically effective approach involves alternating stepsize rules to break the classical gradient alignment. For instance, Dai \cite{Dai2003} proposed the Alternate Minimization (AM) gradient method, which improves performance by alternately minimizing the function value and the gradient norm along the search direction. Building on these concepts, recent contributions have proposed new strategies to further enhance convergence. For example, Sun and Liu \cite{sun2020new} introduced steplength rules based on alternating approximations of the inverse eigenvalues of the Hessian matrix $H$ over the two-dimensional subspace spanned by $\nabla f(\bx_k)$ and $\nabla f(\bx_{k-1})$. Similarly, Zhang and Sun \cite{zhang2024cyclic} analyzed cyclic gradient methods, highlighting the effectiveness of periodic and alternating stepsizes in unconstrained optimization.
In addition to steplength strategies, acceleration has also been achieved through multi-step approaches, where updates incorporate information from past gradients or iterates \cite{dai-3dp, Nesterov:1983, Polyak:1964, yuan-3d}, and further acceleration has been obtained by researchers focusing on composite directions \cite{DeMagistrisCOAP}.

A related line of development can be found in iterative methods for solving linear systems, which have also evolved significantly. A notable example is the Kaczmarz method, a row-action projection algorithm. Recently, Van Lith, Hansen, and Hochstenbach \cite{VanLith2021} introduced a dual-process framework in the context of Kaczmarz iterations. 

Our work draws inspiration directly from the ``Mutual step'' approach in \cite{VanLith2021}, which we adopt and reinterpret in a gradient-based setting, referring to it as the ``Twin Gradient'' method. The core idea is to run two simultaneous processes that cooperate to accelerate convergence by exploiting the geometric relationship between their paths.
Unlike spectral methods that rely on the history of a single sequence to determine the stepsize, our proposed method evolves two parallel sequences of iterates, denoted by $\{\bx_k\}$ and $\{\bz_k\}$. The main idea is that, although they target the same minimum, the two processes follow distinct paths starting from different initializations. At each step, they cooperate to determine their respective stepsizes via a Twin step principle: the steplengths are chosen not to minimize the function values individually, but to minimize the distance between the two processes at the next iteration.

At iteration $k$, both processes use the steepest descent direction. We define the search directions for the first process $\{\bx_k\}$ and the second process $\{\bz_k\}$ as
\[
-\nabla f(\bx_k) \quad \text{and} \quad -\nabla f(\bz_k).
\]
Unlike methods where stepsizes are computed independently, here the steps are coupled to minimize the mutual distance. This geometric coupling generates a search path, quite different from that of a standard gradient method, which can lead to a very fast decrease in the objective functions. Nevertheless, a serious drawback of the Twin strategy occurs when parallelism between the directions arises, which discourages its use as a stand-alone method.
Thus, we embed the Twin strategy into a globalization framework. Specifically, we propose a hybrid algorithm, Twin-ABB$_{\min}$, which employs the Twin phase to rapidly navigate the search space when the processes are distinct, and switches to ABB$_{\min}$ spectral gradient method \cite{Frassoldati:2008} when the processes become collinear or when the mutual acceleration stagnates.

The rest of this paper is organized as follows: Section~\ref{sec:method} introduces the Twin method. Section~\ref{sec:conv} analyses its convergence for general and quadratic functions. Section~\ref{sec:analysis} provides a stepsize analysis. Section~\ref{sec:algo} details the hybrid Twin-ABB$_{\min}$ algorithm and its nonmonotone line search. Finally, numerical experiments and conclusions are presented in Sections~\ref{sec:experiments} and \ref{sec:conclusions}, respectively.

\section{The Twin method}
\label{sec:method}
Our proposed method evolves two sequences of iterates, denoted as $\{\bx_k\}$ and $\{\bz_k\}$. At each iteration $k$, given the current points $\bx_k$ and $\bz_k$ with non-zero gradients, the new iterates are generated by taking steps along their respective steepest descent directions $-\nabla f(\bx_k)$ and $-\nabla f(\bz_k)$. 
We make the very natural assumption that both gradients are non-zero. We now introduce some notation that will be used in the remainder of this work:
\begin{equation} \label{eq:normalized} 
 \bp_k = -\nabla f(\bx_k) \, / \, \|\nabla f(\bx_k)\|, \qquad  \bq_k = -\nabla f(\bz_k) \, / \, \|\nabla f(\bz_k)\|, \qquad
 \bd_k = \bx_k - \bz_k, \qquad \gamma_k = \bp_k^{\top}\bq_k
 \end{equation}
i.e. the normalized gradients at $\bx_k$ and $\bz_k$, the 
the distance vector between the current iterates, and the cosine  of the angle between the two search directions.

The update rules for the two processes are defined as:
\begin{equation}
\bx_{k+1} = \bx_k + \alpha_k \,  \bp_k, \qquad \bz_{k+1} = \bz_k + \beta_k \,  \bq_k,  
\label{updatexz}
\end{equation}
where $\alpha_k \ge 0$ and $\beta_k \ge 0$ are the stepsizes. 
Rather than computing these steplengths independently, as is customary in standard gradient methods, the key innovation in this paper is to \emph{simultaneously determine them by solving a joint minimization problem}.
The aim is to find the stepsizes $\alpha_k$ and $\beta_k$ along the normalized directions that minimize the distance between the two processes at the next iteration, subject to nonnegativity constraints:
\begin{equation}
\label{eq:mutual_step_min}
(\alpha_k, \beta_k) = \argmin_{\alpha, \, \beta \, \ge \, 0} \, \tfrac12 \, \| (\bx_k + \alpha \,  \bp_k) - (\bz_k + \beta \,  \bq_k) \|^2.
\end{equation}
This formulation ensures that the distance between the two processes is a monotonically nonincreasing sequence. Indeed, since the trivial solution $(\alpha, \beta) = (0, 0)$ is always feasible and corresponds to the current distance $\|\bd_k\|$, the constrained minimization ensures $\|\bd_{k+1}\| \le \|\bd_k\|$. 
Let us first consider the corresponding unconstrained problem
\begin{equation}
\label{eq:mutual_step_min_unc}
(\wt\alpha_k, \, \wt\beta_k) = \argmin_{\alpha, \, \beta} \, \tfrac12 \, \| (\bx_k + \alpha \,  \bp_k) - (\bz_k + \beta \,  \bq_k) \|^2.
\end{equation}
Define
\begin{equation} \label{eq:Gk}
G_k = [ \bp_k \ \ \  -\bq_k], \qquad M_k = G_k^\top G_k
= \mtxa{cc}{
1 & -\gamma_k \\
-\gamma_k & 1}.
\end{equation}
It is easy to check that the eigenvalues of $M_k$ are $\lambda_{\min}(M_k) = 1 - |\gamma_k|$ and $\lambda_{\max}(M_k) = 1 + |\gamma_k|$.
As long as the search directions are not parallel at iteration $k$, we have $\lambda_{\min}(M_k) > 0$.
Consequently, $\|M_k^{-1}\| = \frac{1}{1 - |\gamma_k|}$ and for the condition number we have $\kappa(M_k) = \frac{1 + |\gamma_k|}{1 - |\gamma_k|}.$

The numerical solution to the unconstrained problem \eqref{eq:mutual_step_min_unc} is via the least squares problem
$G_k \,
\text{\scalebox{.6}{$\begin{bmatrix} \wt \alpha \\ \wt \beta \end{bmatrix}$}}
\approx -\bd_k$.
The associated normal equations are the symmetric $2\times2$ linear system
$G_k^{\top} G_k \,
\text{\scalebox{.6}{$\begin{bmatrix} \wt \alpha \\ \wt \beta \end{bmatrix}$}}
= - G_k^{\top} \bd_k$, or written in coordinates
\begin{equation}
\label{eq:system2x2_normalized}
\begin{bmatrix}
1 & -\gamma_k \\
-\gamma_k & 1
\end{bmatrix}
\begin{bmatrix}
\wt \alpha \\ \wt \beta
\end{bmatrix}
=
\begin{bmatrix}
- \bp_k^{\top} \bd_k \\[1mm]
\ph{-}  \bq_k^{\top} \bd_k
\end{bmatrix}.
\end{equation}
The solution of the system \eqref{eq:system2x2_normalized} is
\begin{equation}
\label{eq:explicit_steps_normalized}
\wt \alpha_k = (1 - \gamma_k^2)^{-1}
(-\bp_k^{\top} + \gamma_k \, \bq_k^{\top}) \, \bd_k,
\qquad
\wt \beta_k = (1 - \gamma_k^2)^{-1}
(\bq_k^{\top} - \gamma_k \, \bp_k^{\top}) \, \bd_k.
\end{equation}
If ${\wt \alpha}_k \geq 0$ and ${\wt \beta}_k \geq 0$, the solutions of \eqref{eq:mutual_step_min} and  \eqref{eq:mutual_step_min_unc} coincide.
Otherwise, the solution must lie on the boundary of the feasible region. Let us define 
\begin{equation} \label{eq:alfabetacappello}
\begin{array}{rll}
  \wh \alpha_k & = \argmin_{\alpha \, \ge \, 0} \ \tfrac12 \, \| (\bx_k + \alpha \,  \bp_k) - \bz_k \|^2 & = \max (-\bp_k^{\top} \bd_k, \, 0), \\[1.5mm]
  \wh \beta_k  & = \argmin_{\beta \, \ge \, 0} \ \tfrac12 \, \| (\bz_k + \beta \,  \bq_k) - \bx_k \|^2 & = \max (\bq_k^{\top} \bd_k, \, 0).
\end{array}
\end{equation}
Then the solution for \eqref{eq:mutual_step_min} is
\begin{align}
(\alpha_k, \beta_k) =
\begin{cases}
(\wt \alpha_k, \, \wt \beta_k)  & \text{if }  \wt \alpha_k > 0 \text{ and } \wt \beta_k > 0, \\[1.5mm]
(\wh \alpha_k, 0)  & \text{otherwise, if } \|(\bx_k + \wh \alpha_k \, \bp_k) - \bz_k\| \le \|\bx_k - (\bz_k + \wh \beta_k \, \bq_k)\|,\\[1.5mm]
(0, \wh \beta_k)  & \text{otherwise, if } \|(\bx_k + \wh \alpha_k \, \bp_k) - \bz_k\| > \|\bx_k - (\bz_k + \wh \beta_k \, \bq_k)\|.
\end{cases}
\label{eq:final_step_selection}
\end{align}
We note that the update formula \eqref{updatexz} is invariant with respect to positive scaling and additive constants of the objective function $f(x)$, as well as translations of the domain. In our method, $(\alpha_k, \beta_k)$ are not derived from a Rayleigh inverse, nor do we attempt to approximate the Hessian. Indeed, the method is based on the idea of bringing research directions closer together, rather than on the estimated curvature of the function. We designate this iterative scheme, which couples two simultaneous descent paths via the Twin-Step rule, as the Twin Gradient Method. Finally, while the two sequences $\{\bx_k\}$ and $\{\bz_k\}$ are designed to approach the minimizer from distinct geometric directions, to guarantee the maximum decrease, we define the solution at the final iteration as the iterate that achieves the lowest objective function value. The following lemma bounds the stepsizes in terms of the mutual distance and of the angle between the search directions.

\begin{lemma}
\label{lem:stepbound}
Let $\alpha_k$ and $\beta_k$ be the stepsizes for the $k$th iteration computed according to \eqref{eq:final_step_selection}. Then
\begin{align}
| \alpha_k|, \ | \beta_k| \le \frac{\|\bd_k\|}{1 - |\gamma_k|}.
\label{eq:alfabetabnd}
\end{align}
\end{lemma}
\begin{proof}
\hspace*{0.2cm}
If $(\alpha,\beta)=(\tilde\alpha,\tilde\beta)$,
\eqref{eq:alfabetabnd} follows from the fact that $\|-\bp_k + \gamma_k \ \bq_k\| \le 1 + |\gamma_k|$. Otherwise, since
$\wh \alpha_k = \max(-\bp_k^{\top} \bd_k, 0) \le \|\bd_k\|, \; \wh \beta_k = \max(\bq_k^{\top} \bd_k, 0) \le \|\bd_k\|$ and  $1 - |\gamma_k| \le 1$, we have $\|\bd_k\| \le \|\bd_k\|\, / \,(1 - |\gamma_k|)$, and therefore, the inequality \eqref{eq:alfabetabnd} holds.
\end{proof}

The complete procedure of the basic algorithm is summarized in the Algorithm~\ref{alg:TwinBasic}, while we will see its extensions in Section~\ref{sec:algo}. In the choice of steplengths, the use of scaling by a damping factor $\eta_k$ is foreseen to avoid possible overshooting; more about this in Subsection~\ref{subsec:damping}. 

\begin{algorithm}[htb!]
\caption{Twin Gradient Method}
\label{alg:TwinBasic}
\footnotesize
\begin{algorithmic}[1]
    \State \textbf{Input:} $\bx_0, \bz_0 \in \R^n$; ${\sf tol} > 0$; $ {\sf maxit} \in \mathbb{N}$; $\{\eta_k\} \in (0, 1]$
    \For{ $k = 0, \dots, {\sf maxit}$ }
        
        \State Compute normalized gradients $\bp_k$ and $\bq_k$ \eqref{eq:normalized}; \: Compute $(\alpha_k, \beta_k)$ from~\eqref{eq:final_step_selection};
        \State $\bx_{k+1} = \bx_k + \eta_k \, \alpha_k \, \bp_k$, \quad $\bz_{k+1} = \bz_k + \eta_k \, \beta_k \, \bq_k$
        \State {\bf if} {\tt stopping criteria} {\bf return}; \ {\bf end if}
    \EndFor
    \State \textbf{Output:} $\bx^* = \argmin_{\bx \in \{\bx_k, \bz_k\}} f(\bx)$ \Comment{Return the best solution}
\end{algorithmic}

\end{algorithm}

In the sequel we will consider the following assumption: 
\begin{assumption}
\label{ass:gamma_bound}
There exists a constant $\bar{\gamma} \in (0,1)$ such that
$|\gamma_k| \le \bar{\gamma}$ for all $k$.
\end{assumption}

As we shall see, in the convergence analysis the factor $(1-|\gamma_k|)^{-1}$ enters only through the ratio $R_k = \frac{\|\bd_k\|}{1-|\gamma_k|},$ which bounds both stepsizes by Lemma~\ref{lem:stepbound}. We shall therefore consider the following assumption.
\begin{assumption}
\label{ass:Rk}
The sequence $\{R_k\}$ is bounded, i.e., there exists $M>0$ such that $\|\bd_k\| \le M\,(1-|\gamma_k|)$ for all $k$.
\end{assumption}
Assumption~\ref{ass:Rk} has a simple geometric meaning: the search directions may become parallel, i.e., $1-|\gamma_k|$ may tend to zero, but not faster than the two processes approach each other. 

By Lemma~\ref{lem:stepbound} it guarantees bounded stepsizes and, as shown in the sequel, together with $\bd_k \to \zero$ (which holds, e.g., in the quadratic case by Proposition~\ref{cor:Twin_two_conv}) it forces the gradients at both iterates to vanish. Since $\{\|\bd_k\|\}$ is non-increasing for any $\eta_k \in (0,1]$ (see \eqref{eq:mutual_step_min} and Lemma~\ref{prop:damping_contraction}), Assumption~\ref{ass:gamma_bound} implies Assumption~\ref{ass:Rk} with $M = \|\bd_0\|/(1-\gammac)$. Assumption~\ref{ass:Rk} is, however, much weaker, since it allows $|\gamma_k| \to 1$. Some refined statements require the stronger condition $R_k \to 0$; under the Twin geometry and a mild condition on the stepsizes, this is equivalent to requiring that the gradients vanish faster than $\sqrt{\|\bd_k\|}$ (Corollary~\ref{cor:equiv}), and Proposition~\ref{prop:natural} describes a natural regime in which it holds. Finally, in the hybrid method of Section~\ref{sec:algo} the switching test enforces Assumption~\ref{ass:Rk} by construction during the Twin phase (Proposition~\ref{prop:hybrid_twin_phase}).

The linear system~\eqref{eq:system2x2_normalized} admits a clear geometric interpretation. The objective function in~\eqref{eq:mutual_step_min} minimizes  the distance between two points moving along the search lines.
Let $\ell_{\bx}$ be the line passing through $\bx_k$ with direction $\bp_k$, and $\ell_{\bz}$ be the line passing through $\bz_k$ with direction $\bq_k$
\[
\ell_{\bx} = \{ \bx_k + \alpha \, \bp_k \mid \alpha \in \mathbb{R} \}, \quad \ell_{\bz} = \{ \bz_k + \beta \, \bq_k \mid \beta \in \mathbb{R} \}.
\]
The steplengths $(\alpha_k, \beta_k)$ identify the points $\bx_k + \alpha_k \, \bp_k$ and $\bz_k + \beta_k \, \bq_k$ on the lines $\ell_{\bx}$ and $\ell_{\bz}$, respectively, such that $\|\bd_{k+1}\|$ is minimized.
By definition of the gradient descent updates, the new distance vector can be written as $\bd_{k+1} = \bd_k + \alpha_k \, \bp_k - \beta_k \, \bq_k$, so that this vector lies in a three-dimensional subspace: $\bd_{k+1} \in \text{span} \{ \bd_k, \bp_k, \bq_k \}$.
Since $\bd_{k+1} = \bd_k + G_k\,[\alpha_k,\ \beta_k]^\top$, problem~\eqref{eq:mutual_step_min_unc} is the least squares problem $\min_{\by}\tfrac12\|G_k\by+\bd_k\|^2$, whose residual is $\bd_{k+1}$, and the normal equations~\eqref{eq:system2x2_normalized} read $G_k^\top\bd_{k+1}=\zero$. Hence, whenever $(\alpha_k,\beta_k)=(\wt\alpha_k,\wt\beta_k)$, i.e., $\wt\alpha_k>0$ and $\wt\beta_k>0$ in~\eqref{eq:final_step_selection},
\begin{equation}
\label{eq:orth_k+1}
\bp_k^{\top} \, \bd_{k+1} = 0 \quad \text{and} \quad \bq_k^{\top} \, \bd_{k+1} = 0 .
\end{equation}
Geometrically, the shortest segment connecting the lines $\ell_{\bx}$ and $\ell_{\bz}$ is orthogonal to both of them. If the solution of~\eqref{eq:mutual_step_min} lies on the boundary, only one of these relations holds: for $(\alpha_k,\beta_k)=(\wh\alpha_k,0)$ with $\wh\alpha_k=-\bp_k^\top\bd_k>0$ we have $\bp_k^\top\bd_{k+1}=\bp_k^\top\bd_k+\wh\alpha_k=0$, and similarly $\bq_k^\top\bd_{k+1}=0$ for $(0,\wh\beta_k)$ with $\wh\beta_k>0$.
Intuitively, by minimizing the mutual distance between two paths traversing the opposite walls of a narrow valley, the method forces the iterates towards the central floor of the valley, thereby dampening the typical orthogonal oscillations of steepest descent.
To illustrate the practical effects of this geometric coupling, we present a numerical example.
\begin{example}
  We investigate whether the auxiliary process $\bz_k$ can assist the primary process $\bx_k$ in escaping the slow ``zig--zag'' convergence typical of the Steepest Descent (SD) \cite{Cauchy:1847} method on ill-conditioned problems.
We consider a 3D-extension of a classic 2D quadratic problem from Nocedal and Wright \cite[Sec.~3.3]{NocedalWright2006}, known to be challenging for SD:
\[
f(x_1, x_2,x_3) = \tfrac12 \, (x_1^2 + \zeta \, x_2^2 + \zeta \,x_3^2).
\]
The unique minimizer is $\bx^* = (0,0,0)^{\top}$. We set the starting point for the first process as $\bx_0 = (\zeta,\zeta, 1)^{\top}$, $\zeta=1000$ and $\eta_k= \eta = 0.9$.
This configuration forces the standard SD method into a 3D oscillatory convergence pattern, which, in contrast, does not occur with the Twin update.
The SD method  terminates after reaching the maximum number of iterations ($5000$) without satisfying the stopping criteria 
$
\|\nabla f(\bx_k)\| \le 10^{-6} \cdot \|\nabla f(\bx_0)\|.
$
\, For the Twin method, in contrast, we evaluated $100$ runs with randomly chosen starting poi $\bz_0$; to satisfy the stopping criteria the number of iterations required   was $17$, on average, ranging from a minimum of $4$ to a maximum of $40$ iterations in the worst case.  
\end{example}
\section{Convergence Analysis}
\label{sec:conv}
\subsection{About $R_k$}

\medskip
Let’s start by showing how the ratio $R_k$ controls the gradients: Assumption~\ref{ass:Rk}, together with $\bd_k \to \zero$, forces the gradients at both iterates to vanish. We use the following elementary inequality (see, e.g., \cite{maligranda2006simple}).\vspace{0.2cm}
\begin{lemma}
\label{lem:normalization}
For any non-zero vectors $\ba, \bb \in \R^n$,
$\left\| \, \ba \, / \, \|\ba\| - \bb \, / \, \|\bb\| \, \right\| \le 2 \, \|\ba - \bb\| \, / \, \|\ba\|$.
\end{lemma}
\vspace{0.2cm}
\begin{proposition}
\label{prop:key}
Assume that $\nabla f$ is Lipschitz continuous with constant $L>0$, and let $\bx_k$, $\bz_k$ be any two points with non-zero gradients and $|\gamma_k|<1$. Then
\begin{equation}
\label{eq:key}
\|\nabla f(\bx_k)\|^2 \le L^2 \, \|\bd_k\| \, \max\{2R_k, \, \|\bd_k\|\},
\qquad
\|\nabla f(\bz_k)\|^2 \le L^2 \, \|\bd_k\| \, \max\{2R_k, \, \|\bd_k\|\}.
\end{equation}
\end{proposition}
\begin{proof}
\hspace*{0.2cm}
By the Lipschitz continuity of $\nabla f$, $\|\nabla f(\bx_k) - \nabla f(\bz_k)\| \le L\,\|\bd_k\|$.\\
\emph{Case $\gamma_k \ge 0$.} Then $1-|\gamma_k| = 1-\gamma_k = \tfrac12\|\bp_k - \bq_k\|^2$, since $\bp_k$, $\bq_k$ are unit vectors. By Lemma~\ref{lem:normalization} with $\ba = \nabla f(\bx_k)$, $\bb = \nabla f(\bz_k)$,
\begin{equation}
\label{eq:1_gamma}
    1-\gamma_k \le \frac{2\,\|\nabla f(\bx_k) - \nabla f(\bz_k)\|^2}{\|\nabla f(\bx_k)\|^2} \le \frac{2L^2\|\bd_k\|^2}{\|\nabla f(\bx_k)\|^2},
\end{equation}
that is, $\|\nabla f(\bx_k)\|^2 \le 2L^2\|\bd_k\|^2/(1-|\gamma_k|) = 2L^2\|\bd_k\|\,R_k$. Exchanging the roles of $\ba$ and $\bb$ gives the bound for $\bz_k$.\\
\emph{Case $\gamma_k < 0$.} Then
$\|\nabla f(\bx_k) - \nabla f(\bz_k)\|^2 = \|\nabla f(\bx_k)\|^2 + \|\nabla f(\bz_k)\|^2 - 2\|\nabla f(\bx_k)\|\|\nabla f(\bz_k)\|\gamma_k \ge \max\{\|\nabla f(\bx_k)\|^2, \|\nabla f(\bz_k)\|^2\}$, hence $\|\nabla f(\bx_k)\|, \|\nabla f(\bz_k)\| \le L\|\bd_k\|$.
\end{proof}

\begin{corollary}
\label{cor:gradients}
Assume that $\nabla f$ is Lipschitz continuous with constant $L>0$.
\begin{enumerate}
\item If $\bd_k \to \zero$ and Assumption~\ref{ass:Rk} holds, then
\[\|\nabla f(\bx_k)\|, \, \|\nabla f(\bz_k)\| \le L \max\{\sqrt{2M}\,\|\bd_k\|^{1/2}, \, \|\bd_k\|\} = \calo(\|\bd_k\|^{1/2});\] in particular $\nabla f(\bx_k), \nabla f(\bz_k) \to \zero$ and every limit point of $\{\bx_k\}$ and $\{\bz_k\}$ is stationary. For problem~\eqref{eq:quadratic_problem}, $\bx_k \to \bx^*$ and $\bz_k \to \bx^*$, with $\|\bx_k - \bx^*\|, \|\bz_k - \bx^*\| = \calo(\|\bd_k\|^{1/2})$.
\item If $\bd_k \to \zero$ and $R_k \to 0$, then
$\|\nabla f(\bx_k)\|^2 = o(\|\bd_k\|)$ and $\|\nabla f(\bz_k)\|^2 = o(\|\bd_k\|)$. For problem~\eqref{eq:quadratic_problem}, $\|\bx_k - \bx^*\|^2 = o(\|\bd_k\|)$ and $\|\bz_k - \bx^*\|^2 = o(\|\bd_k\|)$.
\item If $R_k \le K\|\bd_k\|^p$ for all $k$, for some $K>0$ and $p \in (0,1]$, then
\[\|\nabla f(\bx_k)\|, \, \|\nabla f(\bz_k)\| \le L \max\{\sqrt{2K}\,\|\bd_k\|^{(1+p)/2}, \, \|\bd_k\|\},\] i.e., $\|\nabla f(\bx_k)\|, \|\nabla f(\bz_k)\| = \calo(\|\bd_k\|^{(1+p)/2})$ as $\bd_k \to \zero$.
\end{enumerate}
\end{corollary}
\begin{proof}
\hspace*{0.2cm}
Item 1: because of $R_k \le M$, inequality \eqref{eq:key} gives \[\|\nabla f(\bx_k)\|^2, \|\nabla f(\bz_k)\|^2 \le L^2\|\bd_k\|\max\{2M, \|\bd_k\|\} \to 0;\] for problem~\eqref{eq:quadratic_problem}, since $\lambda_1\|\bx - \bx^*\| \le \|\nabla f(\bx)\|$, then $\|\bx_k - \bx^*\|, \|\bz_k - \bx^*\| = \calo(\|\bd_k\|^{1/2})$.

\noindent Item 2: by \eqref{eq:key}, $\|\nabla f(\bx_k)\|^2/\|\bd_k\| \le L^2\max\{2R_k, \|\bd_k\|\} \to 0$, and similarly for $\bz_k$.

\noindent Item 3: can be proved by inserting $R_k \le K\|\bd_k\|^p$ in \eqref{eq:key}.
\end{proof}
\begin{remark}
\label{rem:outcome}
Corollary~\ref{cor:gradients} shows that, whenever $\bd_k \to \zero$, the convergence of the Twin iterates is governed by the ratio $R_k$: its boundedness yields the convergence of the gradients to zero, while its decay yields convergence rates (items~2--3). 
\end{remark}

Note that Proposition~\ref{prop:key} and Corollary ~\ref{cor:gradients} use neither convexity nor any property of the Twin iteration: they hold for any pair of sequences $\bx_k$ and $\bz_k$.
\begin{proposition}
\label{prop:parallel}
Assume that $\nabla f$ is Lipschitz continuous with constant $L>0$, that $\bd_k \to \zero$, and that $\|\nabla f(\bx_k)\| \ge c > 0$ for all $k$. Then $\gamma_k \ge 0$ for all $k$ large enough, and
\[
1 - \gamma_k \le \frac{2L^2}{c^2}\,\|\bd_k\|^2 \longrightarrow 0.
\]
In particular $\gamma_k \to 1$, and neither Assumption~\ref{ass:gamma_bound} nor Assumption~\ref{ass:Rk} can hold. 
\end{proposition}
\begin{proof}
\hspace*{0.2cm}
If $\gamma_k < 0$, the proof of Proposition~\ref{prop:key} (\textit{Case $\gamma_k<0$}) gives $\|\nabla f(\bx_k)\| \le L\|\bd_k\| \to 0$, which contradicts the hypothesis. For $\gamma_k \ge 0$, because the inequality \eqref{eq:1_gamma} show in the proof of Proposition~\ref{prop:key}, we get $R_k \ge c^2/(2L^2\|\bd_k\|) \to \infty$ and the thesis holds.
\end{proof}
We stress that $\bd_k \to \zero$ does not by itself force the search directions to become parallel: if both gradients vanish, the angle between them is not constrained. For instance, for $f(\bx) = \frac12\|\bx\|^2$ the points $\bx_k = (1/k,0)^\top$ and $\bz_k = (0,1/k)^\top$ satisfy $\bd_k \to \zero$, while $\gamma_k = 0$ for all $k$. Parallelism is, however, the typical behaviour: by Proposition~\ref{prop:parallel} it is unavoidable whenever the gradients stay bounded away from zero and, more generally, the inequality $1-\gamma_k \le 2L^2\|\bd_k\|^2/\|\nabla f(\bx_k)\|^2$ (see the proof of Proposition~\ref{prop:key}) shows that the directions can remain separated only if the gradients vanish at least as fast as $\|\bd_k\|$. Proposition~\ref{prop:parallel} shows that, for the undamped method applied to problem~\eqref{eq:quadratic_problem}, where $\bd_k \to \zero$ linearly, either the gradients vanish fast enough or the search directions become parallel at a geometric rate. This is consistent with the column $\eta_k=1$ of Table~\ref{tab:comparison}, and it provides a theoretical justification for the switching strategy of Section~\ref{sec:algo}.

\begin{proposition}
\label{prop:general_conv}
Consider Problem \eqref{eq:general_problem}, with $\nabla f$ Lipschitz continuous with constant $L>0$. Let $\{\bx_k\},\;\{\bz_k\}$ be the sequences generated by the Twin method, and suppose that $\bd_k \to \zero$ and that Assumption~\ref{ass:Rk} holds. Then
\[
\|\nabla f(\bx_k)\|, \ \|\nabla f(\bz_k)\| = \calo\big(\|\bd_k\|^{1/2}\big) \to 0,
\]
so that every limit point of $\{\bx_k\}$ and $\{\bz_k\}$ is stationary. If, in addition, $\{\bx_k\}$ is bounded and $f$ has a unique stationary point $\bx^*$ (e.g., $f$ strictly convex with a minimizer), then both processes converge to $\bx^*$; if moreover $\inf_k \eta_k > 0$, then $\alpha_k, \beta_k \to 0$. If $R_k \to 0$, then $\alpha_k, \beta_k \to 0$ in any case, and $\|\nabla f(\bx_k)\|^2, \|\nabla f(\bz_k)\|^2 = o(\|\bd_k\|)$.
\end{proposition}
\begin{proof}
\hspace*{0.2cm}
The bound on the gradients follows from item~1 of Corollary~\ref{cor:gradients}. 
Since a bounded sequence with a single accumulation point converges to that point, then
 $\bx_k \to \bx^*$, and $\bz_k \to \bx^*$ because $\bd_k \to \zero$. Then $\eta_k\alpha_k = \|\bx_{k+1} - \bx_k\| \to 0$ and $\eta_k\beta_k = \|\bz_{k+1} - \bz_k\| \to 0$, so that $\alpha_k, \beta_k \to 0$ if $\inf_k \eta_k > 0$. If $R_k \to 0$, then $\alpha_k, \beta_k \to 0$ by Lemma~\ref{lem:stepbound}, and the $o(\|\bd_k\|)$ bounds follow from item~2 of Corollary~\ref{cor:gradients}.
\end{proof}

\subsection{The convex quadratic case}
\label{subsec:global_convergence}
Let's start by considering the convergence of the twin method for strictly convex quadratic functions with $\eta_k = \eta = 1$.  Our goal is to determine under what conditions   $\bd_k \to \zero$, and that both sequences $\{\bx_k\}$ and $\{\bz_k\}$ converge to the unique minimizer $\bx^*$. We begin by proving that $\bd_k \to \zero$.

A main tool to reach this is a comparison with a modified \textit{one-dimensional} Twin type technique, by adding the constraint $\alpha = \beta$ to \eqref{eq:mutual_step_min}.
We will see in \eqref{eq:BB2_Twin} that the corresponding stepsize is of BB2 appearance (Barzilai--Borwein stepsize 2), but also completely different since it is related to the Twin idea of having two sequences.
For that reason, we use the label ``T2'', and consider
\begin{equation}
\bx_{k+1} = \bx_k +  \beta_k^{\rm{T2}} \,  \bp_k, \qquad \bz_{k+1} = \bz_k +  \beta_k^{\rm{T2}} \,  \bq_k,  
\label{updatexz1}
\end{equation}
with
\begin{align}
 \beta_k^{\rm{T2}} &= \argmin_{\beta \, \ge \, 0} \tfrac12 \, \| (\bx_k - \beta \, \nabla f(\bx_k)) - (\bz_k - \beta \, \nabla f(\bz_k)) \|^2
 = \argmin_{\beta \, \ge \, 0} \tfrac12 \, \| (I - \beta \, A) \, \bd_k \|^2.
 \label{beta1}
\end{align}
If we define
\[
\by_k = \nabla f(\bx_k) - \nabla f(\bz_k) = A \, \bd_k,
\]
the solution of the minimization problem \eqref{beta1} is 
\begin{equation} \label{eq:BB2_Twin}
\beta_k^{\rm{T2}} = \frac{\bd_k^{\top} A \, \bd_k}{\bd_k^{\top} A^2 \, \bd_k} = \frac{\bd_k^{\top} \, \by_k}{\by_k^{\top} \, \by_k}.
\end{equation}
We refer to this approach as the Twin$_{\rm{T2}}$ method. Since $A$ is positive definite, $\bd_k^{\top} A \, \bd_k$ is strictly positive for any $\bd_k \neq 0$, which guarantees $\beta_k^{\rm{T2}} > 0$. Furthermore, the expression in \eqref{eq:BB2_Twin} coincides with the classical BB2 stepsize. In contrast to standard Barzilai--Borwein methods, where $\bx_k - \bx_{k-1}$ plays the role of $\bd_k$, here we have  $\bd_k = \bx_k - \bz_k$. Moreover, the reciprocal steplength admits the Rayleigh quotient representation
\[
(\beta_k^{\rm{T2}})^{-1} = \frac{\bd_k^{\top} A^2 \, \bd_k}{\bd_k^{\top} A \, \bd_k}.
\]
Since $A$ is symmetric positive definite, this can be viewed as a Rayleigh quotient of the vector $A^{1/2} \, \bd_k$, which implies
\[
\lambda_n^{-1} \le \beta_k^{\rm{T2}} \le \lambda_1^{-1}.
\]
The recurrence relation for $\bd_k$ and the minimization problem are identical to the residual update in the classical Minimal Residual (MR) algorithm for symmetric positive definite matrices. Just like \eqref{updatexz}, the update rule \eqref{updatexz1} is invariant under positive scaling and constant additive shifts
of the objective function $f(x)$, as well as under translations of the domain.

We next present a convergence result for the Twin$_{\rm{T2}}$ method.
We stress that although the proof technique of the following result is classical (see, e.g., \cite[p.~118]{Saad2003}), the context is new:
we exploit this results  for $\bd_k = \bx_k -\bz_k$, using the difference of two processes, while the standard application of this result is for $\bx_k - \bx^*$ playing the role of $\bd_k$.

\begin{lemma}
\label{lem:Twin_beta_conv}
Consider problem \eqref{eq:quadratic_problem} with condition number $\kappa(A) = \lambda_n / \lambda_1$.
Consider the sequence $\{\bd_k^{\rm{T2}}\}$ generated by
$\bd_{k+1}^{\rm{T2}} = (I - \beta_k^{\rm{T2}} A)\,\bd_k^{\rm{T2}}$,
where the stepsize is chosen according to \eqref{eq:BB2_Twin}
Then the following inequality holds:
\[
\|\bd_{k+1}^{\rm{T2}}\| \le \frac{\kappa - 1}{\kappa + 1} \cdot \|\bd_k^{\rm{T2}}\|,
\]
and therefore $\bd_k^{\rm{T2}} \to \zero$ at least linearly.
\end{lemma}

\begin{proof}

\hspace*{0.2cm} We provide a proof for completeness.
By definition of $\beta_k^{\rm{T2}}$, for any $\bar{\beta} \in \mathbb{R}$ it holds
\[
\|\bd_{k+1}^{\rm{T2}}\| = \min_{\beta} \|(I - \beta \, A) \, \bd_k^{\rm{T2}}\|
\le \|(I - \bar{\beta} \, A) \, \bd_k^{\rm{T2}}\| \le \|I - \bar{\beta} \, A\| \cdot \|\bd_k^{\rm{T2}}\|.
\]
Since $A$ is symmetric positive definite, the eigenvalues of $I - \bar{\beta} \, A$ are $1 - \bar{\beta} \, \lambda_i$, hence
\[
\|I - \bar{\beta} \, A\| = \max_{1 \le i \le n} |1 - \bar{\beta} \, \lambda_i|.
\]
This quantity is minimal for $\bar{\beta} = \frac{2}{\lambda_1 + \lambda_n}$, for which we have
\[
\|I - \bar{\beta} \, A\| = \tfrac{\lambda_n - \lambda_1}{\lambda_n + \lambda_1}
= \tfrac{\kappa - 1}{\kappa + 1}.
\]
\end{proof}
Using this Lemma, we can now prove the following result for the two-parameter Twin method.

\begin{proposition}
\label{cor:Twin_two_conv}
 Consider the sequence of mutual distances $\{\bd_k\}$ generated by the Twin system \eqref{eq:system2x2_normalized} for problem \eqref{eq:quadratic_problem}. Then $\bd_k \to 0$ as $k \to \infty$.
\end{proposition}
\begin{proof}
\hspace*{0.2cm} Let $\bd_{k+1}$ and $\bd_{k+1}^{\rm{T2}}$ denote the distance vectors generated by the  Twin method and the one-dimensional step $\beta_k^{\rm{T2}}$ (Lemma~\ref{lem:Twin_beta_conv}), respectively. 
Since the 1D search line is a subset of the 2D plane spanned by the search directions, the minimum distance \eqref{eq:mutual_step_min} is upper-bounded by the 1D update. Furthermore, according to Lemma~\ref{lem:Twin_beta_conv}, this one-dimensional update guarantees a contraction by a factor of $\frac{\kappa-1}{\kappa+1} < 1$. Combining these properties, we observe that at every single iteration, the Twin step $\bd_{k+1}$ performs at least as well as the restricted step $\bd_{k+1}^{\rm{T2}}$, yielding the single step-by-step bound:
\[\|\bd_{k+1}\| \le \|\bd_{k+1}^{\rm{T2}}\| \le \tfrac{\kappa-1}{\kappa+1} \cdot \|\bd_k\|,\]
and therefore the thesis follows.
\end{proof}
Proposition \ref{cor:Twin_two_conv} shows that Assumption \ref{ass:gamma_bound} is unrealistic for Algorithm \ref{alg:TwinBasic} with damping factor $\eta_k=1$. In Subsection~\ref{subsec:damping} we discuss whether it can be enforced by a suitable choice of the damping parameters, and at which price (Remark~\ref{rem:open_loop}, Proposition~\ref{prop:closed_loop} and Remark~\ref{rem:lipschitz_limit}).

Note that $\bd_k \to \zero$ only means that the sequences  $\{\bx_k\}$ and $\{\bz_k\}$ are asymptotically close, so they share the same limit if either converges, the same limit points, and the same convergence properties. 

We now present a convergence result for the quadratic case, in which we exploit the Twin system \eqref{eq:system2x2_normalized} for the convergence.
More precisely, the following result shows that for the twin method, the  two sequences convergence to the solution  of \eqref{eq:quadratic_problem} under suitable conditions.  The proposition actually adapts and extends a geometric property originally established for the Twin Kaczmarz method \cite{VanLith2021} to our gradient-based framework.
\vspace{0.2cm}

\begin{proposition}[{Adaptation and extension of \cite[Cor.~4.2]{VanLith2021}}]
\label{thm:asym_orth_conv}
Let $\{\bx_k\}$ and $\{\bz_k\}$ be the sequences generated by Algorithm \ref{alg:TwinBasic} with $\eta_k = 1$ for problem~\eqref{eq:quadratic_problem}, and assume that Assumption~\ref{ass:Rk} holds. Then both processes converge R-linearly to the unique minimizer:
\begin{equation}
\lim_{k \to \infty} \bx_k = \lim_{k \to \infty} \bz_k = \bx^*,
\qquad
\|\bx_k - \bx^*\|, \ \|\bz_k - \bx^*\| = \calo\big(r^{k/2}\big), \quad r = \frac{\kappa-1}{\kappa+1},
\label{xzconvergence}
\end{equation}
and the stepsizes satisfy
\begin{equation}
\label{alfabetato0}
\alpha_k, \ \beta_k = \calo\big(r^{k/2}\big); \quad \text{in particular} \quad \alpha_k \to 0, \ \beta_k \to 0.
\end{equation}
If, in addition, $R_k \to 0$, then $\|\bx_k - \bx^*\|^2 = o(\|\bd_k\|)$ and $\|\bz_k - \bx^*\|^2 = o(\|\bd_k\|)$. If $R_k \le K\|\bd_k\|^p$ for all $k$, for some $K>0$ and $p\in(0,1]$, then $\|\bd_k\| = \calo(\eps)$ implies $\alpha_k, \beta_k = \calo(\eps^p)$, and the rate in \eqref{xzconvergence} improves to $\calo(r^{k(1+p)/2})$.
\end{proposition}
\begin{proof}
\hspace*{0.2cm}
By Proposition \ref{cor:Twin_two_conv}, $\|\bd_k\| \le r^k\|\bd_0\| \to 0$, and \eqref{xzconvergence} follows from item~1 of Corollary~\ref{cor:gradients}. Since $\eta_k = 1$, $\alpha_k \ge 0$ and $\|\bp_k\| = 1$, we have $\alpha_k = \|\bx_{k+1} - \bx_k\| \le \|\bx_{k+1} - \bx^*\| + \|\bx_k - \bx^*\| = \calo(r^{k/2})$, and similarly for $\beta_k$, which gives \eqref{alfabetato0}. The statement under $R_k \to 0$ follows from item~2 of Corollary~\ref{cor:gradients}. Finally, if $R_k \le K\|\bd_k\|^p$, Lemma~\ref{lem:stepbound} gives $|\alpha_k|, |\beta_k| \le R_k \le K\|\bd_k\|^p = \calo(\eps^p)$, and by Proposition~\ref{cor:Twin_two_conv}, $\|\bd_k\| \le r^k\|\bd_0\|$; combine with item 3 and $\|\bx - \bx^*\| \le \|\nabla f(\bx)\|/\lambda_1$.
\end{proof}

We note that Proposition \ref{thm:asym_orth_conv}
strongly relies on the result of Proposition \ref{cor:Twin_two_conv}, which guarantees that the sequence $\bd_k$ generated by the twin gradient method for strictly convex quadratic functions converges to zero. Unfortunately, a similar result does not hold for a generic strictly convex function; in that case, we must assume that the sequence $\bd_k$ converges to the zero vector (see Proposition \ref{prop:general_conv}).
Proposition~\ref{thm:asym_orth_conv} assumes that $\{R_k\}$ is bounded. Still for problem~\eqref{eq:quadratic_problem} and the undamped method, we now show that, under a mild condition on the stepsizes, the necessary condition of Corollary~\ref{cor:gradients} for $R_k\to0$ is also sufficient, and we identify a natural regime in which Assumption~\ref{ass:Rk} holds (Proposition~\ref{prop:natural}).

\begin{lemma}
\label{lem:exact}
Let $\{\bx_k\}$ and $\{\bz_k\}$ be the sequences generated by Algorithm \ref{alg:TwinBasic} with $\eta_k = 1$ for problem~\eqref{eq:quadratic_problem}.
Let $\bu_k = \bd_k/\|\bd_k\|$ and $\theta_k$ the angle between $A\bu_k$ and $\nabla f(\bx_k)$. If $\theta_k \neq 0, \pi$, then
\begin{equation}
\label{eq:exact}
\frac{\|\bd_k\|}{1-\gamma_k} = \frac{(1+\gamma_k)\,\|\nabla f(\bz_k)\|^2}{\|\bd_k\|\,\|A\bu_k\|^2\sin^2\theta_k}.
\end{equation}
In particular, \eqref{eq:exact} gives $R_k$ whenever $\gamma_k \ge 0$.
\end{lemma}
\begin{proof}
\hspace*{0.2cm}
Throughout, $\nabla f(\bx_k)$, $\nabla f(\bz_k)$ and $\bd_k$ are non-zero, so that $\bu_k$ is well defined, and $A\bu_k\neq\zero$ since $A$ is nonsingular.
We first note that if $\bw$ is a unit vector and $P=I-\bw \bw^\top$ is the orthogonal projector onto $\bw^\perp$, then, since $P$ is symmetric and idempotent, for every $\bv \in \R^n$
\begin{equation}
\label{projection}
\| P \bv \|^2 = \bv^\top P\bv = \|\bv\|^2 - (\bv^\top \bw)^2 = \|\bv\|^2 \sin^2 \varphi,
\end{equation}
where $\varphi$ is the angle between $\bv$ and $\bw$.
Let $P_k^\perp$ be the orthogonal projector onto $\nabla f (\bx_k)^\perp$. Since $\gamma_k$ is the cosine of the angle between $\nabla f(\bx_k)$ and $\nabla f(\bz_k)$, \eqref{projection} gives $\|P_k^\perp \nabla f(\bz_k)\|^2 = \|\nabla f(\bz_k)\|^2(1 - \gamma_k^2)$. On the other hand, $\nabla f(\bz_k) = \nabla f(\bx_k) - A\bd_k$ and $P_k^\perp\nabla f(\bx_k)=\zero$, so that $P_k^\perp \nabla f(\bz_k) = -P_k^\perp A\bd_k$; applying \eqref{projection} to $A\bd_k = \|\bd_k\|\,A\bu_k$ we get $\|P_k^\perp A\bd_k\|^2 = \|\bd_k\|^2 \| A\bu_k\|^2 \sin^2\theta_k$. Therefore
\begin{equation}
\label{projection1}
\|\nabla f(\bz_k)\|^2(1 - \gamma_k^2) = \|\bd_k\|^2 \| A\bu_k\|^2 \sin^2\theta_k .
\end{equation}
Since $\theta_k \neq 0,\pi$, the right-hand side is positive, hence $|\gamma_k|<1$. Writing $1-\gamma_k^2=(1-\gamma_k)(1+\gamma_k)$ dividing \eqref{projection1} by $(1+\gamma_k)\,\|\nabla f(\bz_k)\|^2$ and taking reciprocals, multiplied by $\|\bd_k\|$, gives \eqref{eq:exact}.
\end{proof}

\begin{lemma}
\label{lem:twin_orth}
If $\eta_k = 1$ and $\wt\alpha_k > 0$, $\wt\beta_k > 0$, then $\nabla f(\bx_k)^\top\bd_{k+1} = 0$ and $\nabla f(\bz_k)^\top\bd_{k+1} = 0$, i.e.,
$A\bd_{k+1} \perp \bx_k - \bx^*$ and $A\bd_{k+1} \perp \bz_k - \bx^*$.
\end{lemma}
\begin{proof}
\hspace*{0.2cm}
This is \eqref{eq:orth_k+1}, together with $\nabla f(\bx_k)^\top\bd_{k+1} = (\bx_k - \bx^*)^\top A\bd_{k+1}$ and $\nabla f(\bz_k)^\top\bd_{k+1} = (\bz_k - \bx^*)^\top A\bd_{k+1}$.
\end{proof}

\begin{proposition}
\label{prop:theta}
Under the assumptions of Lemma~\ref{lem:twin_orth}, assume that $\bd_{k+1}\neq\zero$ and $\nabla f(\bx_{k+1})\neq\zero$, and let $\tau_k = \nabla f(\bx_k)^\top A^{-1}\nabla f(\bx_k)/\|\nabla f(\bx_k)\|$ and $\be_k = \bx_k -\bx^*$. Then
\begin{equation}
\label{eq:theta}
\sin\theta_{k+1} \ge \frac{\|\nabla f(\bx_k)\|\,|\tau_k - \alpha_k|}{\|\be_k\|\,\|\nabla f(\bx_{k+1})\|}.
\end{equation}
Moreover, $\tau_k \ge \|\nabla f(\bx_k)\|^3/(\nabla f(\bx_k)^\top A\nabla f(\bx_k))$, the exact steepest descent (Cauchy) step along $\bp_k$, and if $\alpha_k \le (1-c)\,\tau_k$ for some $c \in (0,1)$, then
\[
\sin\theta_{k+1} \ge \frac{c}{\kappa(1+\kappa)}.
\]
\end{proposition}
\begin{proof}
\hspace*{0.2cm}
Let $\bw = A\bu_{k+1}/\|A\bu_{k+1}\|$, $\bv = \nabla f(\bx_{k+1})/\|\nabla f(\bx_{k+1})\|$ and $\wh\be_k = \be_k/\|\be_k\|$. By Lemma~\ref{lem:twin_orth}, $\wh\be_k \perp \bw$, hence
\[
\sin\theta_{k+1} = \|\bv - (\bw^\top\bv)\bw\| \ge \big|\wh\be_k^\top\big(\bv - (\bw^\top\bv)\bw\big)\big| = |\wh\be_k^\top\bv| = \frac{|\be_k^\top A\be_{k+1}|}{\|\be_k\|\,\|\nabla f(\bx_{k+1})\|}.
\]
Since $\be_{k+1} = \be_k - \alpha_k\nabla f(\bx_k)/\|\nabla f(\bx_k)\|$, $\be_k^\top A\nabla f(\bx_k) = \|\nabla f(\bx_k)\|^2$ and $\be_k^\top A\be_k = \nabla f(\bx_k)^\top A^{-1}\nabla f(\bx_k)$, we get $\be_k^\top A\be_{k+1} = \|\nabla f(\bx_k)\|(\tau_k - \alpha_k)$, which proves \eqref{eq:theta}. The bound on $\tau_k$ is the Cauchy--Schwarz inequality $(\nabla f(\bx_k)^\top\nabla f(\bx_k))^2 \le (\nabla f(\bx_k)^\top A\nabla f(\bx_k))(\nabla f(\bx_k)^\top A^{-1}\nabla f(\bx_k))$.
If $\alpha_k \le (1-c)\tau_k$, then $|\tau_k - \alpha_k| \ge c\,\tau_k$; moreover $\alpha_k \le \tau_k \le \|\nabla f(\bx_k)\|/\lambda_1$, so $\|\nabla f(\bx_{k+1})\| \le \|\nabla f(\bx_k)\| + \lambda_n\alpha_k \le (1+\kappa)\|\nabla f(\bx_k)\|$. Using also $\tau_k \ge \|\nabla f(\bx_k)\|/\lambda_n$ and $\|\nabla f(\bx_k)\| \ge \lambda_1\|\be_k\|$,
\[
\sin\theta_{k+1} \ge \frac{c\,\tau_k}{(1+\kappa)\|\be_k\|} \ge \frac{c\,\|\nabla f(\bx_k)\|}{\lambda_n(1+\kappa)\|\be_k\|} \ge \frac{c}{\kappa(1+\kappa)}. \qedhere
\]
\end{proof}

\begin{corollary}
\label{cor:equiv}
Assume that, for all $k$ large enough, $\eta_k = 1$, $\wt\alpha_k > 0$, $\wt\beta_k > 0$, $\alpha_k \le (1-c)\tau_k$ and $\gamma_{k+1} \ge 0$. Then, with $s = c/(\kappa(1+\kappa))$,
\[
\frac{1}{\lambda_n^2}\,\frac{\|\nabla f(\bz_{k+1})\|^2}{\|\bd_{k+1}\|} \le R_{k+1} \le \frac{2}{\lambda_1^2 s^2}\,\frac{\|\nabla f(\bz_{k+1})\|^2}{\|\bd_{k+1}\|}.
\]
Consequently, if $\bd_k \to \zero$,
\[
R_k \to 0 \iff \|\nabla f(\bz_k)\|^2 = o(\|\bd_k\|) \iff \|\nabla f(\bx_k)\|^2 = o(\|\bd_k\|).
\]
\end{corollary}
\begin{proof}
\hspace*{0.2cm}
The bounds follow from \eqref{eq:exact}, with $1 \le 1+\gamma_{k+1} \le 2$ and $\lambda_1 \le \|A\bu_{k+1}\| \le \lambda_n$, and from Proposition~\ref{prop:theta}. The last equivalence follows from $\big|\|\nabla f(\bx_k)\| - \|\nabla f(\bz_k)\|\big| \le \lambda_n\|\bd_k\|$ and $\|\bd_k\|^2 = o(\|\bd_k\|)$.
\end{proof}

\begin{proposition}
\label{prop:natural}
Consider problem~\eqref{eq:quadratic_problem} under the assumptions of Corollary~\ref{cor:equiv}. If there exists $C>0$ such that $\|\bz_k - \bx^*\| \le C\,\|\bd_k\|$ for all $k$ large enough, then $R_k = \calo(\|\bd_k\|)$. In particular, Assumption~\ref{ass:Rk} holds, and $R_k \to 0$ whenever $\bd_k \to \zero$.
\end{proposition}
\begin{proof}
\hspace*{0.2cm}
By Corollary~\ref{cor:equiv} and $\|\nabla f(\bz_{k+1})\| \le \lambda_n\|\bz_{k+1}-\bx^*\|$,
\[
R_{k+1} \le \frac{2}{\lambda_1^2 s^2}\,\frac{\|\nabla f(\bz_{k+1})\|^2}{\|\bd_{k+1}\|}
\le \frac{2\lambda_n^2 C^2}{\lambda_1^2 s^2}\,\|\bd_{k+1}\|. \qedhere
\]
\end{proof}
In words, $R_k \to 0$ (and hence Assumption~\ref{ass:Rk}) holds naturally as long as the two processes remain distinguishable at the scale of their distance from the minimizer, i.e., $\bd_k$ does not collapse faster than $\bz_k - \bx^*$; in this regime $R_k$ decays as fast as under Assumption~\ref{ass:gamma_bound}. We recall that Assumption~\ref{ass:gamma_bound} also implies Assumption~\ref{ass:Rk}, and even $R_k \to 0$ when $\bd_k \to \zero$.

\begin{remark}
\label{rem:finite_termination}
An extreme case is that of finite termination. If $\bx_k - \bx^*$ and $\bz_k - \bx^*$ are both eigenvectors of $A$ (for instance, if $A = cI$), then each gradient is parallel to the corresponding error, so that both search lines $\ell_{\bx}$ and $\ell_{\bz}$ pass through $\bx^*$. If they are not parallel, the Twin step yields $\bx_{k+1} = \bz_{k+1} = \bx^*$ in a single iteration.
\end{remark}

\begin{remark}
Since $A^{-1}\nabla f(\bx_k) = \be_k$, $\sin\theta_k \to 0$ means that $\bd_k$ aligns with the Newton direction $\bx_k - \bx^*$, i.e., $\bz_k$ approaches the line through $\bx_k$ and $\bx^*$, on which $\nabla f(\bz_k)$ is parallel to $\nabla f(\bx_k)$. If $\bz_k$ lies on the same side of $\bx^*$ as $\bx_k$, then $\gamma_k \to 1$ and $R_k$ may fail to go to zero; if it lies on the opposite side, then $\gamma_k \to -1$. Lemma~\ref{lem:twin_orth} shows that the Twin step prevents the first situation as long as the stepsizes do not overshoot $\tau_k$.
\end{remark}

The Twin structure, through Lemma~\ref{lem:twin_orth}, is essential: for general sequences the necessary condition of Corollary~\ref{cor:gradients} is not sufficient, as the following example shows.

\begin{example}
\label{ex:not_sufficient}
Let $A = \operatorname{diag}(1,2)$, $\bb = \zero$, $\bx_k = (1/k, \, 0)^\top$ and $\bz_k = (1/(2k), \, -1/k^m)^\top$ with $m > 3/2$ (these sequences are not generated by the Twin method). Then $\|\bd_k\| \sim 1/(2k)$ and $\|\nabla f(\bx_k)\|^2 = 1/k^2 = o(\|\bd_k\|)$. However, $\bz_k$ approaches the segment between $\bx_k$ and $\bx^*$, so that $\sin\theta_k \to 0$, and
\[
1 - \gamma_k \sim 8\,k^{2-2m}, \qquad R_k \sim \frac{k^{2m-3}}{16} \to +\infty.
\]
\end{example}

\begin{remark}
\label{rem:limitations}
The analysis above covers the regime $\gamma_k \ge 0$, where $1-|\gamma_k| = 1-\gamma_k$. If $\gamma_k \to -1$ (opposite gradients, with $\bx_k$ and $\bz_k$ on opposite sides of $\bx^*$), then $R_k = \|\bd_k\|/(1+\gamma_k)$, and one needs separately $1+\gamma_k \gg \|\bd_k\|$; in that regime $\|\nabla f(\bx_k)\|, \|\nabla f(\bz_k)\| \le \lambda_n\|\bd_k\|$. Moreover, Lemma~\ref{lem:twin_orth} requires $\eta_k = 1$ and positive unconstrained steps. With damping (Subsection~\ref{subsec:damping}), $\bd_{k+1}(\eta) = (1-\eta)\bd_k + \eta\,\bd_{k+1}(1)$, so that $\nabla f(\bx_k)^\top\bd_{k+1}(\eta) = (1-\eta)\,\nabla f(\bx_k)^\top\bd_k$, and \eqref{eq:theta} holds only up to a perturbation proportional to $1-\eta_k$; if one of the steps lies on the boundary, only one of the two orthogonality relations holds.
\end{remark}

\subsection{The use of a damping sequence}
\label{subsec:damping}

As we have seen, with $\eta_k=1$, in the quadratic case, the sequence $\{\bd_k\}$ converges to the null vector and, by Proposition~\ref{prop:parallel}, the search directions typically become parallel. A sequence of damping factors $\eta_k\in(0,1]$ slows down, or even prevents, the collapse of the two processes, and our computational experience suggests that this may offer practical advantages. The extreme case is $n=2$, for which the Twin method without damping terminates after the first iteration. Damping plays two distinct roles in this paper, which it is useful to distinguish from the outset.
\begin{itemize}
\item \emph{Damping as an algorithmic option.} In the hybrid methods of Section~\ref{sec:algo} (Algorithms~\ref{alg:TwinABB} and~\ref{alg:Twin_nonquad}) the damping factors are an optional ingredient of the Twin phase, and the numerical experiments of Section~\ref{sec:experiments} consider both the damped and the undamped versions. The damping used there, $\eta_k = 0.9 + 0.1\,\frac{k-1}{k+1}$, is non-decreasing and tends to $1$, and is therefore covered by Propositions~\ref{thm:conv_damped1} and~\ref{thm:conv_damped2} below. Moreover, the convergence of the Twin phase of the hybrid methods (Proposition~\ref{prop:hybrid_twin_phase}) holds for any choice of $\eta_k\in(0,1]$, since it relies only on the switching test.
\item \emph{Damping as a theoretical device to enforce Assumption~\ref{ass:gamma_bound}.} This is discussed at the end of this subsection: a Lipschitz-based choice of $\eta_k$ keeps the search directions separated but prevents convergence (Remark~\ref{rem:open_loop}); a closed-loop variant removes this obstruction (Proposition~\ref{prop:closed_loop}), but only with short steps (Remark~\ref{rem:lipschitz_limit}).
\end{itemize}
First, let us define the damped updates as

\begin{equation}
\label{eq:damped_updates}
\bx_{k+1} = \bx_{k+1}(\eta_k) = \bx_k + \eta_k \, \alpha_k \, \bp_k \quad \text{and} \quad \bz_{k+1}=\bz_{k+1}(\eta_k) = \bz_k + \eta_k \, \beta_k \, \bq_k,
\end{equation}
where $\alpha_k$ and $\beta_k$ are the stepsizes given by \eqref{eq:final_step_selection}, and $\eta_k \in (0,1]$ is the damping factor. Furthermore, let $\gamma_{k+1}(\eta) = \bp_{k+1}(\eta)^{\top} \bq_{k+1}(\eta)$ denote the cosine of the angle between the updated search directions, where \[\bp_{k+1}(\eta) = -\nabla f(\bx_{k+1}(\eta))/\|\nabla f(\bx_{k+1}(\eta))\|  \quad \text{and} \quad \bq_{k+1}(\eta) = -\nabla f(\bz_{k+1}(\eta))/\|\nabla f(\bz_{k+1}(\eta))\|.\] Before proceeding, it is important to verify that introducing a damping factor $\eta$ preserves the monotonic decrease of the mutual distance.
\vspace*{0.1cm}
\begin{lemma}
\label{prop:damping_contraction}
For any damping factor $\eta \in (0, 1)$ the damped updates  satisfy
\[\|\bx_{k+1}(\eta) - \bz_{k+1}(\eta)\| < \|\bx_k - \bz_k\|.\]
\end{lemma}
\begin{proof}
\hspace*{0.2cm} By the minimization property of the  Twin step \eqref{eq:mutual_step_min}, assuming a non-zero update, we know
\[
\|\bx_{k+1}(1) - \bz_{k+1}(1)\| < \|\bx_k - \bz_k\|.
\]
Therefore, in view of
\[
\| \bx_{k+1}(1) - \bz_{k+1}(1)\|^2 = \| \bx_k - \bz_k\|^2 + \| \alpha_k \, \bp_k - \beta_k \, \bq_k \|^2 +2 \, (\bx_k - \bz_k)^\top \, (\alpha_k \, \bp_k - \beta_k \, \bq_k),
\]
we have
\[
\|\alpha_k \, \bp_k - \beta_k \, \bq_k\|^2 + 2 \, (\bx_k - \bz_k)^{\top} (\alpha_k \, \bp_k - \beta_k \, \bq_k) < 0.
\]
Since
\[
\| \bx_{k+1}(\eta) - \bz_{k+1}(\eta)\|^2 = \| \bx_k - \bz_k\|^2 + \eta^2 \, \| \alpha_k \, \bp_k - \beta_k \, \bq_k \|^2 +2 \eta \, (\bx_k - \bz_k)^\top \, (\alpha_k \, \bp_k - \beta_k \, \bq_k ),
\]
and, because $0 < \eta < 1$,
\[
\begin{array}{ll}
 \eta^2 \, \| \alpha_k \, \bp_k - \beta_k \, \bq_k \|^2 +2 \eta \, (\bx_k - \bz_k)^\top \, (\alpha_k \, \bp_k - \beta_k \, \bq_k ) \\[1.5mm]
\ph{MMMMMMMM} < \eta \, \| \alpha_k \, \bp_k - \beta_k \, \bq_k \|^2 +2 \eta \, (\bx_k - \bz_k)^\top \, (\alpha_k \, \bp_k - \beta_k \, \bq_k )<0,
\end{array}
\]
and the result follows.
\end{proof}

\begin{proposition}
\label{thm:conv_damped1}
Consider problem~\eqref{eq:quadratic_problem}
and the sequence $\{ \bd_k \}$ generated by Algorithm \ref{alg:TwinBasic}, with  $\{ \eta_k \}$ non-decreasing and positive sequence such that
$\eta_k \rightarrow 1$.
If Assumption~\ref{ass:Rk} holds, then $\| \bd_k \| \to 0.$
\end{proposition}
\begin{proof}
\hspace*{0.2cm}
Let $R_k \le M$ for all $k$.
Note that, because of Proposition~\ref{cor:Twin_two_conv} and Lemma~\ref{lem:stepbound},
\begin{align}
\| \bd_{k+1} \| &= \| \bx_{k+1}(\eta_k) - \bz_{k+1}(\eta_k) \| =
\| \bx_{k+1}(1) - \bz_{k+1}(1) - (1-\eta_k)(\alpha_k\bp_k-\beta_k\bq_k)\| \nonumber \\
&\leq \frac{\kappa -1}{\kappa + 1} \| \bd_{k} \| + (1-\eta_k)(|\alpha_k| + |\beta_k|) \nonumber \\
&\leq \frac{\kappa -1}{\kappa + 1} \| \bd_{k}\| + 2\,(1-\eta_k)\,R_k \le \frac{\kappa -1}{\kappa + 1} \| \bd_{k}\| + 2\,(1-\eta_k)\,M, \label{case0}
\end{align}

Setting $r = \frac{\kappa -1}{\kappa + 1} < 1$ and $\varepsilon_k = 2(1-\eta_k)M$, we obtain the recurrence $\| \bd_{k+1} \| \le r \| \bd_{k} \| + \varepsilon_k$. Since $\eta_k \to 1$ implies $\varepsilon_k \to 0$, standard sequence properties yield $\| \bd_k \| \to 0$.
\end{proof}

We can now generalize the main convergence result from Proposition \ref{thm:asym_orth_conv}.  

\begin{proposition}
\label{thm:conv_damped2}
Consider problem~\eqref{eq:quadratic_problem} and the sequences $\{\bx_k\}$ and $\{\bz_k\}$ generated by Algorithm \ref{alg:TwinBasic}, with $\{\eta_k\}$ non-decreasing and positive sequence such that $\eta_k \rightarrow 1$. If Assumption~\ref{ass:Rk} holds, then $\alpha_k, \beta_k \to 0$ and
\begin{equation}
\label{xkzktosol_damp}
\bx_k \to \bx^*, \qquad \bz_k \to \bx^*, \qquad \text{with} \quad \|\bx_k - \bx^*\|, \ \|\bz_k - \bx^*\| = \calo\big(\|\bd_k\|^{1/2}\big),
\end{equation}
where $\bx^*$ is the solution of the problem \eqref{eq:quadratic_problem}.
\end{proposition}
\begin{proof}
\hspace*{0.2cm}
By Proposition~\ref{thm:conv_damped1}, $\bd_k \to \zero$, and \eqref{xkzktosol_damp} follows from item~1 of Corollary~\ref{cor:gradients}. Since $\{\eta_k\}$ is non-decreasing and positive, $\eta_k \ge \eta_0 > 0$, so that $\alpha_k \le \|\bx_{k+1} - \bx_k\|/\eta_0 \to 0$, and similarly $\beta_k \to 0$.
\end{proof}

Propositions~\ref{thm:conv_damped1} and~\ref{thm:conv_damped2} cover the damping sequence used in the experiments of Section~\ref{sec:experiments}. Note, however, that they still require Assumption~\ref{ass:Rk}: a damping with $\eta_k\to1$ slows down the alignment of the search directions, but does not prevent it. By Proposition~\ref{prop:parallel}, for the undamped method this alignment is in fact unavoidable unless the gradients vanish fast enough.

We now examine whether a suitable damping can \emph{enforce} Assumption~\ref{ass:gamma_bound}, i.e., keep $\{|\gamma_k|\}$ uniformly away from $1$. To this aim it is convenient to allow, more generally, two damping factors $\eta_k^{\bx},\eta_k^{\bz}\in(0,1]$, i.e., $\bx_{k+1} = \bx_k + \eta_k^{\bx}\alpha_k\bp_k$ and $\bz_{k+1} = \bz_k + \eta_k^{\bz}\beta_k\bq_k$ ($\eta_k^{\bx}=\eta_k^{\bz}=\eta_k$ gives \eqref{eq:damped_updates}), and to introduce the \emph{effective steplengths}
\begin{equation}
\label{eq:eff_steps}
t_k^{\bx} = \frac{\|\bx_{k+1}-\bx_k\|}{\|\nabla f(\bx_k)\|} = \frac{\eta_k^{\bx}\alpha_k}{\|\nabla f(\bx_k)\|}, \qquad
t_k^{\bz} = \frac{\|\bz_{k+1}-\bz_k\|}{\|\nabla f(\bz_k)\|} = \frac{\eta_k^{\bz}\beta_k}{\|\nabla f(\bz_k)\|},
\end{equation}
so that $\bx_{k+1} = \bx_k - t_k^{\bx}\nabla f(\bx_k)$ and $\bz_{k+1} = \bz_k - t_k^{\bz}\nabla f(\bz_k)$: each damped Twin step is a steepest descent step, and the Twin principle only selects its length. The following perturbation bound is the basic tool.

\begin{lemma}
\label{lem:gamma_perturb}
Assume that $\nabla f$ is Lipschitz continuous with constant $L>0$, and that the gradients at $\bx_k,\bz_k,\bx_{k+1},\bz_{k+1}$ are non-zero. Then
\begin{equation}
\label{eq:double1}
|\gamma_{k+1}-\gamma_k| \le \|\bp_{k+1}-\bp_k\| + \|\bq_{k+1}-\bq_k\| \le 2L\,\big(t_k^{\bx}+t_k^{\bz}\big).
\end{equation}
\end{lemma}
\begin{proof}
\hspace*{0.2cm}
We have $|\gamma_{k+1}-\gamma_k| = |\bp_{k+1}^\top\bq_{k+1}-\bp_k^\top\bq_k| \le |\bp_{k+1}^\top(\bq_{k+1}-\bq_k)| + |(\bp_{k+1}-\bp_k)^\top\bq_k|$, which gives the first inequality, since all vectors involved are unit vectors. By Lemma~\ref{lem:normalization} with $\ba=\nabla f(\bx_k)$, $\bb=\nabla f(\bx_{k+1})$, and by the Lipschitz continuity of $\nabla f$,
$\|\bp_{k+1}-\bp_k\| \le 2\|\nabla f(\bx_{k+1})-\nabla f(\bx_k)\|/\|\nabla f(\bx_k)\| \le 2L\|\bx_{k+1}-\bx_k\|/\|\nabla f(\bx_k)\| = 2L\,t_k^{\bx}$, and similarly for $\bq$.
\end{proof}

\begin{remark}[Open-loop damping]
\label{rem:open_loop}
A first way to exploit \eqref{eq:double1} is to decide in advance how much $|\gamma_k|$ may grow at each iteration. With a single damping factor $\eta_k^{\bx}=\eta_k^{\bz}=\eta_k$, the effective steplengths \eqref{eq:eff_steps} are $t_k^{\bx}=\eta_k\alpha_k/\|\nabla f(\bx_k)\|$ and $t_k^{\bz}=\eta_k\beta_k/\|\nabla f(\bz_k)\|$, and \eqref{eq:double1} reads
\begin{equation}
\label{eq:gamma_single}
|\gamma_{k+1}-\gamma_k| \le 2L\,\eta_k\Big( \frac{\alpha_k}{\|\nabla f(\bx_k)\|} + \frac{\beta_k}{\|\nabla f(\bz_k)\|} \Big).
\end{equation}
The initial gap between $|\gamma_0|$ and $1$ is $1-|\gamma_0|$. We allow $|\gamma_k|$ to use at most half of it over the whole run, distributed over the iterations as a geometric sequence,
\[
\eps_k = (1-|\gamma_0|)\big(\tfrac12\big)^{k+1}, \quad k\ge1, \qquad\text{so that}\qquad \sum_{k\ge1}\eps_k = \tfrac12(1-|\gamma_0|),
\]
and at iteration $k$ we require $|\gamma_{k+1}-\gamma_k|\le\eps_{k+1}$. By \eqref{eq:gamma_single} this holds as soon as $2L\eta_k(\alpha_k/\|\nabla f(\bx_k)\| + \beta_k/\|\nabla f(\bz_k)\|)\le\eps_{k+1}$, and taking the largest such factor we obtain
\begin{equation}
\label{eq:epsk}
\eta_k = \min \Big\{\frac{\eps_{k+1}}{2L}\,\Big( \frac{\alpha_k}{\|\nabla f(\bx_k)\|} + \frac{\beta_k}{\|\nabla f(\bz_k)\|} \Big)^{-1}, \ \hat\gamma\Big\},
\qquad \hat\gamma = \tfrac12(1+|\gamma_0|).
\end{equation}
The second term in the minimum only ensures $\eta_k<1$; it is part of the rule used for Table~\ref{tab:comparison} but plays no role in the argument. Then $|\gamma_{k+1}| \le |\gamma_k| + \eps_{k+1}$, and by induction
\[
|\gamma_k| \le |\gamma_0| + \sum_{i=1}^{k}\eps_i < |\gamma_0| + \tfrac12(1-|\gamma_0|) = \hat\gamma < 1 \qquad\text{for all } k,
\]
so that Assumption~\ref{ass:gamma_bound} holds with $\gammac=\hat\gamma$. Any summable sequence with $\sum_k\eps_k<1-|\gamma_0|$ would serve the same purpose; the geometric choice simply places $\hat\gamma$ halfway between $|\gamma_0|$ and $1$.
Table~\ref{tab:comparison} illustrates the effect on a quadratic function with $n=5$ and $L=100$, with $\bx_0$ generated randomly and $\bz_0$ chosen so that the initial gradients are orthogonal: the damping slows down the alignment of the two directions whereas, without damping, the alignment occurs quickly, resulting in large stepsizes and divergence.

The price is high. By construction $t_k^{\bx},t_k^{\bz}\le \eps_{k+1}/(2L)$, and $\{\eps_k\}$ is summable, so the iterates travel a finite distance. More precisely, let $\bx^*$ be a stationary point, $D=\|\bx_0-\bx^*\|$ and $T_k = \sum_{j<k}\|\bx_{j+1}-\bx_j\|$. Since $\|\nabla f(\bx_j)\| \le L\|\bx_j-\bx^*\| \le L(D+T_k)$ for $j<k$ and $\sum_{j\ge0}\eps_{j+1} = (1-|\gamma_0|)/2$, we get $T_k \le (1-|\gamma_0|)(D+T_k)/4$, i.e.,
\[
\sum_{k=0}^{\infty}\|\bx_{k+1}-\bx_k\| \le \frac{1-|\gamma_0|}{3+|\gamma_0|}\,D,
\qquad
\|\bx_k-\bx^*\| \ge \frac{2(1+|\gamma_0|)}{3+|\gamma_0|}\,D \quad\text{for all } k,
\]
and the same holds for $\{\bz_k\}$ with $D=\|\bz_0-\bx^*\|$. Hence, if $\bx_0\neq\bx^*$, the sequence $\{\bx_k\}$ does not converge to $\bx^*$: it converges to a point $\bar\bx$ which is in general not stationary, consistently with the slow decrease of $f_k$ in the last iterations of the damped run in Table~\ref{tab:comparison}. Moreover, $\bd_k$ does not tend to zero (unless $\bx_0=\bar\bx$): otherwise, since Assumption~\ref{ass:gamma_bound} implies Assumption~\ref{ass:Rk}, item~1 of Corollary~\ref{cor:gradients} would give $\nabla f(\bar\bx)=\zero$, contradicting the bound above with $\bx^*=\bar\bx$. Finally, the factors \eqref{eq:epsk} are not of the type considered in Propositions~\ref{thm:conv_damped1} and~\ref{thm:conv_damped2}.
\end{remark}

The obstruction in Remark~\ref{rem:open_loop} is not the requirement $|\gamma_k|\le\gammac$ itself, but the fact that the budget is fixed a priori and is consumed by every step, also when $|\gamma_k|$ decreases. Indeed, what is needed for convergence is that the effective steplengths \eqref{eq:eff_steps} are not summable, rather than the damping factors. This suggests the following closed-loop rule, in which the budget is recomputed at every iteration from the current value of $|\gamma_k|$ and the two processes are damped independently. Given $\gammac\in(|\gamma_0|,1)$, $c\in(0,1)$ and $\nu\in(0,1]$, set
\begin{subequations}
\label{eq:closed_loop}
\begin{equation}
\delta_k^{\gamma} = c\,(\gammac-|\gamma_k|), \qquad \bar t_k = \frac{\delta_k^{\gamma}}{4L}, 
\end{equation}
\begin{equation}
t_k^{\bx} = \min\Big\{\max\Big(\frac{\alpha_k}{\|\nabla f(\bx_k)\|},\,\nu\,\bar t_k\Big),\,\bar t_k\Big\}, \qquad
t_k^{\bz} = \min\Big\{\max\Big(\frac{\beta_k}{\|\nabla f(\bz_k)\|},\,\nu\,\bar t_k\Big),\,\bar t_k\Big\}.
\end{equation}
\end{subequations}
In words, the Twin steplength is accepted whenever it lies in $[\nu\bar t_k,\bar t_k]$ and is clipped to this interval otherwise. When $\alpha_k/\|\nabla f(\bx_k)\| \ge \nu\bar t_k$ this is a damping with $\eta_k^{\bx} = \min\{1,\bar t_k\|\nabla f(\bx_k)\|/\alpha_k\}$; the lower clip replaces Twin steps that are too small (possibly zero, see the boundary cases in \eqref{eq:final_step_selection}).

\begin{proposition}
\label{prop:closed_loop}
Assume that $\nabla f$ is Lipschitz continuous with constant $L>0$ and that $|\gamma_0|<\gammac<1$, and let $\{\bx_k\}$, $\{\bz_k\}$ be generated, with non-zero gradients at all iterates, by $\bx_{k+1} = \bx_k - t_k^{\bx}\nabla f(\bx_k)$, $\bz_{k+1} = \bz_k - t_k^{\bz}\nabla f(\bz_k)$ with \eqref{eq:closed_loop}. Then:
\begin{enumerate}
\item[(i)] $|\gamma_k| < \gammac$ for all $k$, i.e., Assumption~\ref{ass:gamma_bound} holds, and $\gammac-|\gamma_{k+1}| \ge (1-c)(\gammac-|\gamma_k|)$;
\item[(ii)] $\nu\,\delta_k^{\gamma}/(4L) \le t_k^{\bx},t_k^{\bz} \le 1/(4L)$, and $f(\bx_{k+1}) \le f(\bx_k) - \tfrac78\,t_k^{\bx}\|\nabla f(\bx_k)\|^2$, and similarly for $\{\bz_k\}$;
\item[(iii)] if $f$ is bounded below, then $\sum_k \delta_k^{\gamma}\|\nabla f(\bx_k)\|^2 < \infty$ and $\sum_k \delta_k^{\gamma}\|\nabla f(\bz_k)\|^2 < \infty$;
\item[(iv)] if $f$ is convex with a minimizer $\bx^*$, then
\[
f(\bx_K)-f(\bx^*) \le \frac{2L\,\|\bx_0-\bx^*\|^2}{\nu\sum_{k<K}\delta_k^{\gamma}},
\]
and similarly for $\{\bz_k\}$.
\end{enumerate}
Consequently, either $\sum_k \delta_k^{\gamma} = \infty$, in which case, if $f$ is bounded below, $\liminf_k\|\nabla f(\bx_k)\| = \liminf_k\|\nabla f(\bz_k)\| = 0$ and, for convex $f$, $f(\bx_k),f(\bz_k)\to f(\bx^*)$; or $|\gamma_k|\to\gammac$ with $\sum_k(\gammac-|\gamma_k|)<\infty$.
\end{proposition}
\begin{proof}
\hspace*{0.2cm}
Throughout, $\gamma_k$ is well defined, since all gradients are non-zero (otherwise a stationary point has been found), and Lemma~\ref{lem:gamma_perturb} applies at every iteration.

(i) By \eqref{eq:closed_loop}, $t_k^{\bx},t_k^{\bz}\le\bar t_k$, so that Lemma~\ref{lem:gamma_perturb} gives
\[
|\gamma_{k+1}-\gamma_k| \le 2L\,(t_k^{\bx}+t_k^{\bz}) \le 4L\,\bar t_k = \delta_k^{\gamma} = c\,(\gammac-|\gamma_k|).
\]
Hence $|\gamma_{k+1}| \le |\gamma_k| + |\gamma_{k+1}-\gamma_k| \le |\gamma_k| + c(\gammac-|\gamma_k|)$, i.e., $\gammac-|\gamma_{k+1}| \ge (1-c)(\gammac-|\gamma_k|)$. Since $\gammac-|\gamma_0|>0$, by induction $\gammac-|\gamma_k| \ge (1-c)^k(\gammac-|\gamma_0|)>0$ for all $k$.

(ii) Since $\nu\le1$, the clipping in \eqref{eq:closed_loop} gives $\nu\bar t_k\le t_k^{\bx},t_k^{\bz}\le\bar t_k$, and $\bar t_k = c(\gammac-|\gamma_k|)/(4L) \le c\,\gammac/(4L) < 1/(4L)$. By the descent lemma, $f(\by) \le f(\bx) + \nabla f(\bx)^\top(\by-\bx) + \tfrac L2\|\by-\bx\|^2$, applied with $\bx=\bx_k$ and $\by=\bx_{k+1}=\bx_k-t_k^{\bx}\nabla f(\bx_k)$,
\[
f(\bx_{k+1}) \le f(\bx_k) - t_k^{\bx}\Big(1-\tfrac{L t_k^{\bx}}{2}\Big)\|\nabla f(\bx_k)\|^2 \le f(\bx_k) - \tfrac78\,t_k^{\bx}\|\nabla f(\bx_k)\|^2,
\]
since $Lt_k^{\bx}\le 1/4$. The same argument applies to $\{\bz_k\}$.

(iii) Summing the inequality in (ii) for $k=0,\dots,K-1$ gives $\tfrac78\sum_{k<K} t_k^{\bx}\|\nabla f(\bx_k)\|^2 \le f(\bx_0)-f(\bx_K) \le f(\bx_0)-\inf f$. Using $t_k^{\bx}\ge\nu\delta_k^{\gamma}/(4L)$ and letting $K\to\infty$,
\[
\sum_{k}\delta_k^{\gamma}\|\nabla f(\bx_k)\|^2 \le \frac{32L}{7\nu}\,\big(f(\bx_0)-\inf f\big) < \infty .
\]

(iv) Let $f^*=f(\bx^*)$. Since $t_k^{\bx}\le 1/L$, the first inequality in (ii) gives $f(\bx_{k+1}) \le f(\bx_k) - \tfrac{t_k^{\bx}}{2}\|\nabla f(\bx_k)\|^2$, and by convexity $f(\bx_k) \le f^* + \nabla f(\bx_k)^\top(\bx_k-\bx^*)$. Adding the two inequalities and expanding $\|\bx_{k+1}-\bx^*\|^2 = \|\bx_k-\bx^*-t_k^{\bx}\nabla f(\bx_k)\|^2$, we get
\[
f(\bx_{k+1})-f^* \le \nabla f(\bx_k)^\top(\bx_k-\bx^*) - \tfrac{t_k^{\bx}}{2}\|\nabla f(\bx_k)\|^2 = \frac{1}{2t_k^{\bx}}\big(\|\bx_k-\bx^*\|^2-\|\bx_{k+1}-\bx^*\|^2\big).
\]
Multiplying by $t_k^{\bx}$ and summing over $k<K$, the right-hand side telescopes: $\sum_{k<K} t_k^{\bx}\,(f(\bx_{k+1})-f^*) \le \tfrac12\|\bx_0-\bx^*\|^2$. Since $\{f(\bx_k)\}$ is non-increasing by (ii), $f(\bx_K)\le f(\bx_{k+1})$ for $k<K$, hence $(f(\bx_K)-f^*)\sum_{k<K}t_k^{\bx} \le \tfrac12\|\bx_0-\bx^*\|^2$, and (iv) follows from $t_k^{\bx}\ge\nu\delta_k^{\gamma}/(4L)$.

Finally, if $\sum_k\delta_k^{\gamma}=\infty$, then (iii) forces $\liminf_k\|\nabla f(\bx_k)\|=0$ (otherwise $\|\nabla f(\bx_k)\|\ge\epsilon>0$ for all large $k$ and the series in (iii) would diverge), and (iv) gives $f(\bx_K)\to f^*$ for convex $f$. If instead $\sum_k\delta_k^{\gamma}<\infty$, then $\sum_k(\gammac-|\gamma_k|)=c^{-1}\sum_k\delta_k^{\gamma}<\infty$, and in particular $|\gamma_k|\to\gammac$. The same holds for $\{\bz_k\}$.
\end{proof}

Compared with \eqref{eq:epsk}, the budget $\delta_k^{\gamma}$ is regenerated whenever $|\gamma_k|$ moves away from $\gammac$, and the two processes are damped independently, so that the a priori obstruction of Remark~\ref{rem:open_loop} disappears. Note also that Proposition~\ref{prop:closed_loop} does not use any property of the Twin steplengths $\alpha_k,\beta_k$: they enter only through the clipping in \eqref{eq:closed_loop}, so that the result is, in essence, a statement about steepest descent with steplengths in $[\nu\bar t_k,\bar t_k]$ (see also Remark~\ref{rem:lipschitz_limit}). The only remaining failure mode, $|\gamma_k|\to\gammac$, can be detected during the iterations by a test of the type $|\gamma_k|>\gamma^*$ with $\gamma^*<\gammac$, analogous to the switching test of the hybrid methods of Section~\ref{sec:algo}. Finally, item (i) gives only $\delta_k^{\gamma}\ge c(1-c)^k(\gammac-|\gamma_0|)$, so the second alternative cannot be excluded by this argument alone.

\begin{remark}[An intrinsic limitation]
\label{rem:lipschitz_limit}
Any damping rule whose control of $\gamma_k$ is certified only through \eqref{eq:double1}, in the sense that $|\gamma_{k+1}-\gamma_k|\le\delta_k$ is guaranteed by enforcing $2L(t_k^{\bx}+t_k^{\bz})\le\delta_k$ for some $\delta_k\le2$, yields $t_k^{\bx},t_k^{\bz}\le\delta_k/(2L)\le1/L$. This applies both to \eqref{eq:epsk} and to \eqref{eq:closed_loop}. Certifying the angle through the Lipschitz constant alone therefore turns the Twin method into a short-step steepest descent method, whose rate is governed by $\kappa(A)$ in the quadratic case, and rules out the long steps responsible for the fast initial decrease observed for the Twin phase. This is the reason why, in the methods of Section~\ref{sec:algo}, the separation of the search directions is monitored by the switching test rather than enforced by damping. Whether a damping strategy, possibly based on measured rather than worst-case information on $\gamma_{k+1}(\eta)$ or on curvature information along the search directions, can enforce Assumption~\ref{ass:gamma_bound} while preserving the long Twin steps remains an open question.
\end{remark}

\section{Stepsize analysis} \label{sec:analysis}
We will now present some results regarding stepsizes, limiting ourselves to the undamped case ($\eta_k = 1$).
We begin by presenting an upper bound for the stepsizes.

Let us analyze the asymptotic situation for general objective functions. The following proposition provides additional information about the behavior of the two sequences $\alpha_k$ and $\beta_k$.
Without loss of generality, we may assume that the sequence converges to the origin, so that $\bx^* = \zero$.
Consider the asymptotic regime where both $\|\bx_k\|$ and $\|\bz_k\|$ are $\calo(\eps)$, which implies $\|\bd_k\| = \|\bx_k - \bz_k\| = \calo(\eps)$.

\begin{proposition}
\label{thm:quad_sync}
Assume that the gradient $\nabla f$ is Lipschitz continuous with constant $L > 0$.
Let the sequences $\{\bx_k\}$ and $\{\bz_k\}$ converge to a minimizer $\bx^*$. If $\|\bx_k-\bx^*\| = \calo(\eps)$ and $\|\bz_k - \bx^*\| = \calo(\eps)$, then $| \alpha_k - \beta_k| = \calo(\eps^2)$.
\end{proposition}
\begin{proof}
From \eqref{eq:final_step_selection}, the difference of the stepsizes is given by:
\begin{equation*}
\alpha_k - \beta_k =
\begin{cases}
\wt\alpha_k - \wt\beta_k = -(1 + \gamma_k)^{-1} \, (\bp_k + \bq_k)^{\top} \bd_k \: \: \: \: \text{if }  \wt \alpha_k > 0 \text{ and } \wt \beta_k > 0,\\[1.5mm]
 \wh\alpha_k = -\bp_k^{\top} \bd_k \: \: \: \:\text{otherwise, if } \|(\bx_k + \wh \alpha_k \, \bp_k) - \bz_k\| \le \|\bx_k - (\bz_k + \wh \beta_k \, \bq_k)\|,\\[1.5mm]
 \wh\beta_k = \bq_k^{\top} \bd_k  \: \: \: \: \: \: \: \: \text{otherwise, if } \|(\bx_k + \wh \alpha_k \, \bp_k) - \bz_k\| > \|\bx_k - (\bz_k + \wh \beta_k \, \bq_k)\|,
\end{cases}
\end{equation*}
The asymptotic assumptions imply that $\|\bd_k\| = \calo(\eps)$ and  $\|\bx_k - \bx_{k-1}\| = \calo(\eps)$.
By the geometric construction of the method, the exact orthogonality condition \eqref{eq:orth_k+1} shifted to the current iteration yields $\bp_{k-1}^{\top} \bd_k = 0$ and $\bq_{k-1}^{\top} \bd_k = 0$. We know that $\bp_k = \bp_{k-1} + \nabla^2 f(\bxi) \, (\bx_k - \bx_{k-1}),$ for some $\bxi$ between $\bx_{k-1}$ and $\bx_k$. Therefore, $\bp_k = \bp_{k-1} + \bw$, where $\|\bw\| = \|\nabla^2 f(\bxi) \, (\bx_k-\bx_{k-1})\| \le L \, \eps$ and so $\|\bp_k-\bp_{k-1}\|$ and $\|\bq_k-\bq_{k-1}\|$ are both $\calo(\eps)$. 
Since $\|\bw\| = \calo(\eps)$ and $\|\bd_k\| = \calo(\eps)$, we obtain $|\bp_k^{\top} \bd_k| = \calo(\eps^2)$.
A similar observation holds for $\bq_k$, yielding $|\bq_k^{\top} \bd_k| = \calo(\eps^2)$, and then 

\begin{equation*}
|\alpha_k - \beta_k| =
\begin{cases}
|\wt\alpha_k - \wt\beta_k| \le (1+\gamma_k)^{-1} \big( |\bp_k^{\top} \bd_k| + |\bq_k^{\top} \bd_k| \big) = \calo(\eps^2) \: \: \: \: \text{if }  \wt \alpha_k > 0 \text{ and } \wt \beta_k > 0.\\[1.5mm]
\wh\alpha_k \le |\bp_k^{\top} \bd_k| = \calo(\eps^2) \: \: \: \: \text{otherwise, if } \|(\bx_k + \wh \alpha_k \, \bp_k) - \bz_k\| \le \|\bx_k - (\bz_k + \wh \beta_k \, \bq_k)\|. \\[1.5mm]
 \wh\beta_k \le |\bq_k^{\top} \bd_k| = \calo(\eps^2) \: \: \: \: \text{otherwise, if } \|(\bx_k + \wh \alpha_k \, \bp_k) - \bz_k\| > \|\bx_k - (\bz_k + \wh \beta_k \, \bq_k)\|.
\end{cases}
\end{equation*}
\end{proof}
While Proposition~\ref{thm:asym_orth_conv} guarantees, if $R_k = \calo(\|\bd_k\|^{1/2})$, that the individual stepsizes scale as $\alpha_k, \,  \beta_k = \calo(\sqrt{\eps})$, Proposition~\ref{thm:quad_sync} reveals that their difference decays quadratically. This suggests the two search processes synchronize faster than the sequences themselves converge to the optimum. 
To empirically validate this theoretical result we consider the same strictly convex quadratic problem used in Section~\ref{subsec:global_convergence}. Table~\ref{tab:alpha_beta} shows the evolution of the stepsizes confirming that $\alpha_k-\beta_k \rightarrow 0$ faster than the single steps.

\begin{table}[h!]
    \centering
    \caption{Twin method with and without damping: convergence history ($f_k=\min \{f(\mathbf{x}_k),f(\mathbf{z}_k)\}$).}
    \label{tab:comparison}
    \begin{tabular}{r c c c c c c} 
        \hline
        $k$ & \multicolumn{3}{c}{$\eta_k = 1$} & \multicolumn{3}{c}{$\eta_k$ via \eqref{eq:epsk}} \\
        \cline{2-4} \cline{5-7}
        & $|\gamma_k|$ & $f_k$ & $\|\bx_k - \bx_{k-1}\|$ & $|\gamma_k|$ & $f_k$ & $\|\bx_k - \bx_{k-1}\|$ \\
        \hline
        1 & $4.85 \cdot 10^{-1}$ & 221418 & $2.48 \cdot 10^{2}$ & 0.344 & 221418 & $1.24 \cdot 10^{2}$ \\
        2 & $9.03 \cdot 10^{-1}$ & 43998 & $6.71 \cdot 10^{1}$ & 0.073 & 66435 & $1.11 \cdot 10^{0}$ \\
        3 & $9.92 \cdot 10^{-1}$ & 5286 & $4.29 \cdot 10^{1}$ & 0.182 & 65284 & $6.73 \cdot 10^{1}$ \\
        4 & $9.99 \cdot 10^{-1}$ & 34022 & $4.23 \cdot 10^{1}$ & 0.026 & 20525 & $3.66 \cdot 10^{1}$ \\
        5 & $9.99 \cdot 10^{-1}$ & 12509 & $8.32 \cdot 10^{1}$ & 0.188 & 7002 & $2.11 \cdot 10^{1}$ \\
        6 & $9.99 \cdot 10^{-1}$ & 226586 & $1.20 \cdot 10^{2}$ & 0.189 & 2467 & $1.25 \cdot 10^{1}$ \\
        7 & $9.99 \cdot 10^{-1}$ & 141127 & $2.58 \cdot 10^{2}$ & 0.038 & 883 & $4.16 \cdot 10^{0}$ \\
        8 & $9.99 \cdot 10^{-1}$ & 2091589 & $3.68 \cdot 10^{2}$ & 0.133 & 499 & $9.23 \cdot 10^{-1}$ \\
        9 & $9.99 \cdot 10^{-1}$ & 1334328 & $7.91 \cdot 10^{2}$ & 0.139 & 430 & $3.92 \cdot 10^{-1}$ \\
        10 & $9.99 \cdot 10^{-1}$ & 19676631 & $1.13 \cdot 10^{3}$ & 0.141 & 403 & $1.82 \cdot 10^{-1}$ \\
        \hline
    \end{tabular}
\end{table}

\begin{table}[h!]
    \centering
    \caption{Undamped Twin Method: convergence history of $\{ \alpha_k\}$ and $\{ \beta_k\}$.} 
    \label{tab:alpha_beta}
    \begin{tabular}{r c c c} 
        \hline
        $k$ & $\alpha_k$ & $\beta_k$ & $\alpha_k-\beta_k$ \\
        \hline
        1  & $3.14 \cdot 10^{-2}$ & $0$ & $3.14 \cdot 10^{-2}$ \\
        2  & $1.23 \cdot 10^{-2}$ & $4.38 \cdot 10^{-1}$ & $4.26 \cdot 10^{-1}$ \\
        3  & $3.40 \cdot 10^{-2}$ & $5.17 \cdot 10^{-2}$ & $1.77 \cdot 10^{-2}$ \\
        4  & $1.93 \cdot 10^{-2}$ & $1.97 \cdot 10^{-2}$ & $3.92 \cdot 10^{-4}$ \\
        5  & $3.37 \cdot 10^{-2}$ & $3.37 \cdot 10^{-2}$ & $3.51 \cdot 10^{-5}$ \\
        6  & $2.06 \cdot 10^{-2}$ & $2.06 \cdot 10^{-2}$ & $7.94 \cdot 10^{-8}$ \\
        7  & $3.88 \cdot 10^{-2}$ & $3.88 \cdot 10^{-2}$ & $5.57 \cdot 10^{-7}$ \\
        8  & $1.88 \cdot 10^{-2}$ & $1.88 \cdot 10^{-2}$ & $4.82 \cdot 10^{-8}$ \\
        9  & $4.86 \cdot 10^{-2}$ & $4.86 \cdot 10^{-2}$ & $5.97 \cdot 10^{-8}$ \\
        10 & $1.68 \cdot 10^{-2}$ & $1.68 \cdot 10^{-2}$ & $4.30 \cdot 10^{-9}$ \\
        \hline
    \end{tabular}
\end{table}

We shall now restrict our analysis to the strictly convex quadratic problem \eqref{eq:quadratic_problem}, for which we shall provide further specific properties of the stepsizes. 
\begin{proposition}
Consider the quadratic problem \ref{eq:quadratic_problem}.
Let $\bx_k, \bz_k$ be two iterates with non-zero gradients and let assume ${\alpha}_k, {\beta}_k$ solve the Twin system \eqref{eq:system2x2_normalized}.
If $|\gamma_k| < 1$, then
\[
\|\bn_k\| \cdot (1+|\gamma_k|)^{-1} \cdot \lambda_n^{-1} \le \sqrt{\smash[b]{{\alpha}_k^2 + {\beta}_k^2}} \le \|\bn_k\| \cdot (1-|\gamma_k|)^{-1} \cdot \lambda_1^{-1},
\]
where $\bn_k = [\|\nabla f(\bx_k)\|, \, \|\nabla f(\bz_k)\|]^{\top}$.
\end{proposition}

\begin{proof}
\hspace*{0.2cm}
Let $G_k = [\bp_k \ \ -\bq_k]$ as defined in \eqref{eq:Gk}. The matrix of the unconstrained Twin system \eqref{eq:system2x2_normalized} is $M_k = G_k^{\top} G_k$.
Let ${\by} = [{\alpha}_k, {\beta}_k]^{\top}$.
For the right-hand side in the system $M_k {\by} = \br$ we have
\[
\br = \smtxa{r}{ - \bp_k^{\top} (\bx_k - \bz_k) \\[1mm] \bq_k^{\top} (\bx_k - \bz_k)} = - G_k^{\top} (\bx_k - \bz_k)
\]
We know that $\nabla f(\bx_k) = A\bx_k - \bb$ and $\nabla f(\bz_k) = A\bz_k - \bb$. Subtracting these yields:
\[
\nabla f(\bz_k) - \nabla f(\bx_k) = -A \, (\bx_k - \bz_k) \implies \bx_k - \bz_k = -A^{-1} \, (\nabla f(\bz_k) - \nabla f(\bx_k)).
\]
By definition of $\bp_k$ and $\bq_k$, this difference can be written using $G_k$:
\[
\nabla f(\bz_k) - \nabla f(\bx_k) = \|\nabla f(\bx_k)\| \, \bp_k - \|\nabla f(\bz_k)\| \, \bq_k = G_k \, \bn_k.
\]
Consequently, the distance vector is $\bx_k - \bz_k = -A^{-1} G_k \, \bn_k$. Substituting this back into the expression for $\br$ yielding
\[
\br = G_k^{\top} A^{-1} G_k \, \bn_k.
\]
Therefore, the stepsize vector is given by
\[
{\by} = M_k^{-1} G_k^{\top} A^{-1} G_k \, \bn_k.
\]
Since $|\gamma_k| < 1$, both $M_k^{-1}$ and $G_k^{\top} \! A^{-1} G_k$ are symmetric positive definite, and therefore their singular values are equal to their eigenvalues.
Therefore we have:
\[
\|{\by}\| \le \|M_k^{-1}\| \cdot \|G_k^{\top} \! A^{-1} G_k\| \cdot \|\bn_k\| \le (1 - |\gamma_k|)^{-1} \cdot \lambda_1^{-1} \cdot \|\bn_k\|.
\]
Similarly, for the lower bound, we have
\[
\|\by\| \ge \min_i \, |\lambda_i^{-1}(M_k)| \cdot \min_i \, |\lambda_i^{-1}(A)| \cdot \|\bn_k\|
= (1 + |\gamma_k|)^{-1} \cdot \lambda_n^{-1} \cdot \|\bn_k\|.
\]
\end{proof}
According to the previous proposition we have $\sqrt{\smash[b]{{\alpha}_k^2 + {\beta}_k^2}} = \Theta(\eps)$: the individual stepsizes decay linearly with the distance,  while their mutual difference vanishes at an accelerated quadratic rate $\calo(\eps^2)$. The bounds established in the previous proposition rely on the norm of the vector $\bn_k = [\|\nabla f(\bx_k)\|, \|\nabla f(\bz_k)\|]^{\top}$. The following corollary clarifies the behavior of $\|\bn_k\|$.

\begin{corollary}
\label{cor:n_asymptotic}
Let $\bx^*$ be the solution of problem \eqref{eq:quadratic_problem} and let $\bn_k = [\|\nabla f(\bx_k)\|, \|\nabla f(\bz_k)\|]^{\top}$. 
If $\|\bx_k - \bx^*\| = \Theta(\eps)$ and $\|\bz_k - \bx^*\| = \Theta(\eps)$, then $\|\bn_k\| = \Theta(\eps)$.
\end{corollary}
\begin{proof}
\hspace*{0.2cm}
We know that 
\[
\nabla f(\bx_k) = A\bx_k - \bb = A \, (\bx_k - \bx^*), \quad \text{and} \quad \nabla f(\bz_k) = A\bz_k - \bb = A \, (\bz_k - \bx^*),
\]
and then 
\[
\lambda_1 \, \|\bx_k - \bx^*\| \le \|\nabla f(\bx_k)\| \le \, \lambda_n \, \|\bx_k - \bx^*\|.
\]
This implies that $\|\nabla f(\bx_k)\| = \Theta(\|\bx_k - \bx^*\|)$, and similarly $\|\nabla f(\bz_k)\| = \Theta(\|\bz_k - \bx^*\|)$.
By definition of $\bn_k$,
\[
\|\bn_k\| = \sqrt{\smash[b]{\|\nabla f(\bx_k)\|^2 + \|\nabla f(\bz_k)\|^2}}.
\]
Therefore, if both sequences are at a distance $\Theta(\eps)$ from $\bx^*$, their gradient norms are also $\Theta(\eps)$, directly yielding $\|\bn_k\| = \Theta(\eps)$.
\end{proof}

We next analyze the properties of the solutions $\wt\alpha_k$ and $\wt\beta_k$ of the problem \eqref{eq:mutual_step_min_unc}. We will show a result for the orthogonal case, as this is both a practically relevant situation (we will enforce this in the beginning of the process; see Section~\ref{sec:z0}) and allows for a relatively straightforward analysis.
In that case the function in \eqref{eq:mutual_step_min_unc} becomes additively separable, decoupling into the sum of a function of $\alpha$ and a function of $\beta$. 
The following formulas therefore apply to the steplengths
\[
\wt\alpha_k = - \bp_k^{\top} \, (\bx_k - \bz_k), \qquad
\wt\beta_k = \bq_k^{\top} \, (\bx_k - \bz_k).
\]
Assume without loss of generality $f(\bz_k) \ge f(\bx_k)$.
By Taylor expansion one has
\[
\begin{array}{rll}
\ph{-}\wt\alpha_k \, \|\nabla f(\bx_k)\| & = \nabla f(\bx_k)^{\top} \, (\bx_k-\bz_k) & = f(\bx_k)-f(\bz_k)+\tfrac12 \, (\bx_k-\bz_k)^{\top}A \, (\bx_k-\bz_k),\\[1.5mm]
-\wt\beta_k \, \|\nabla f(\bz_k)\| & = \nabla f(\bz_k)^{\top} \, (\bx_k-\bz_k) & = f(\bx_k)-f(\bz_k)-\tfrac12 \, (\bx_k-\bz_k)^{\top}A \, (\bx_k-\bz_k),  
\end{array}
\]
and therefore
\begin{equation}
\begin{array}{l}
\wt\beta_k > 0, \label{beta} \\
 \wt\alpha_k \ \Bigl\{ 
 \begin{array}{ccc}
 >0    & \textrm{if} & 
 f(\bz_k) -f(\bx_k) \ge
 \tfrac12 \, (\bx_k-\bz_k)^{\top}A \, (\bx_k-\bz_k), \\
 \le 0    & \textrm{otherwise}.
\end{array}
\end{array}
\end{equation}
Figure \ref{fig:cases} depicts the two possible cases that may arise in \eqref{beta}.

\begin{figure}[htbp]
    \centering
    \includegraphics[width=0.4\linewidth]{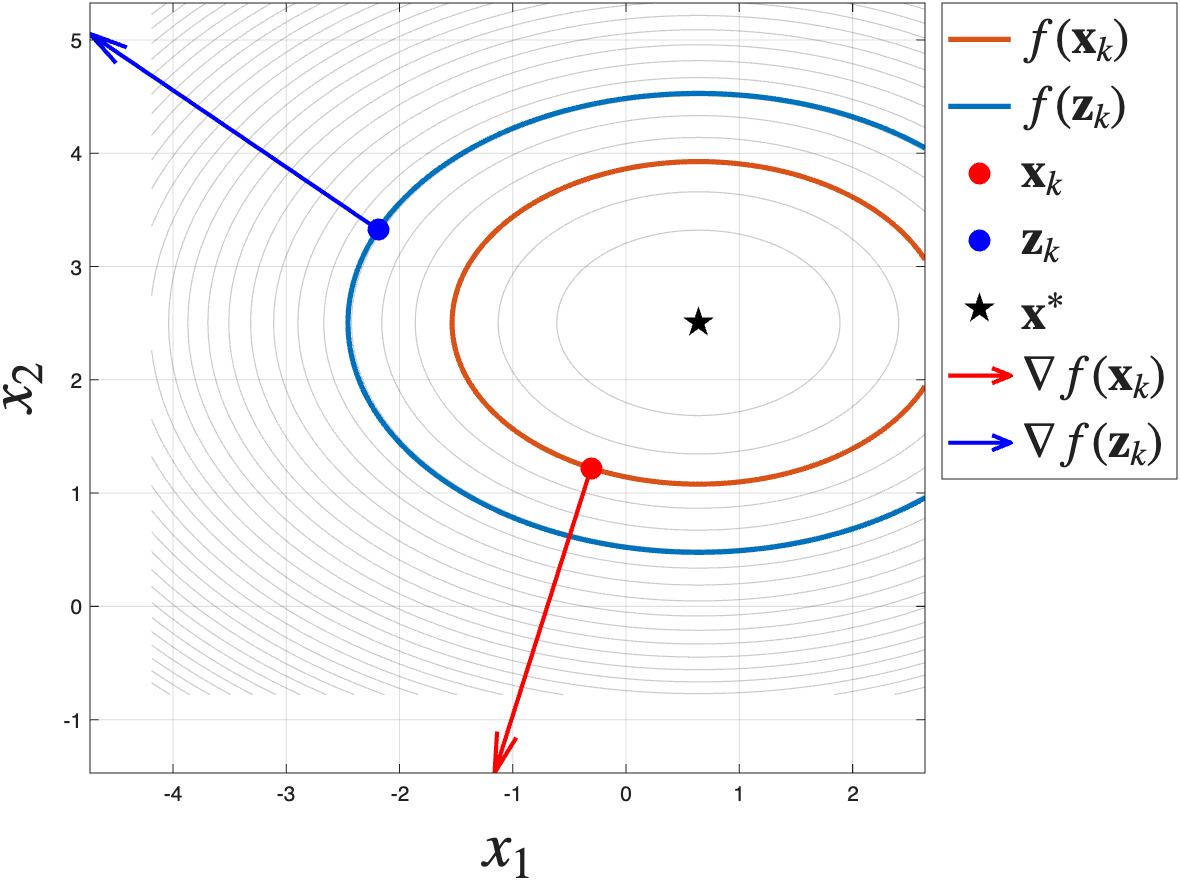}
    \hfill
    \includegraphics[width=0.4\linewidth]{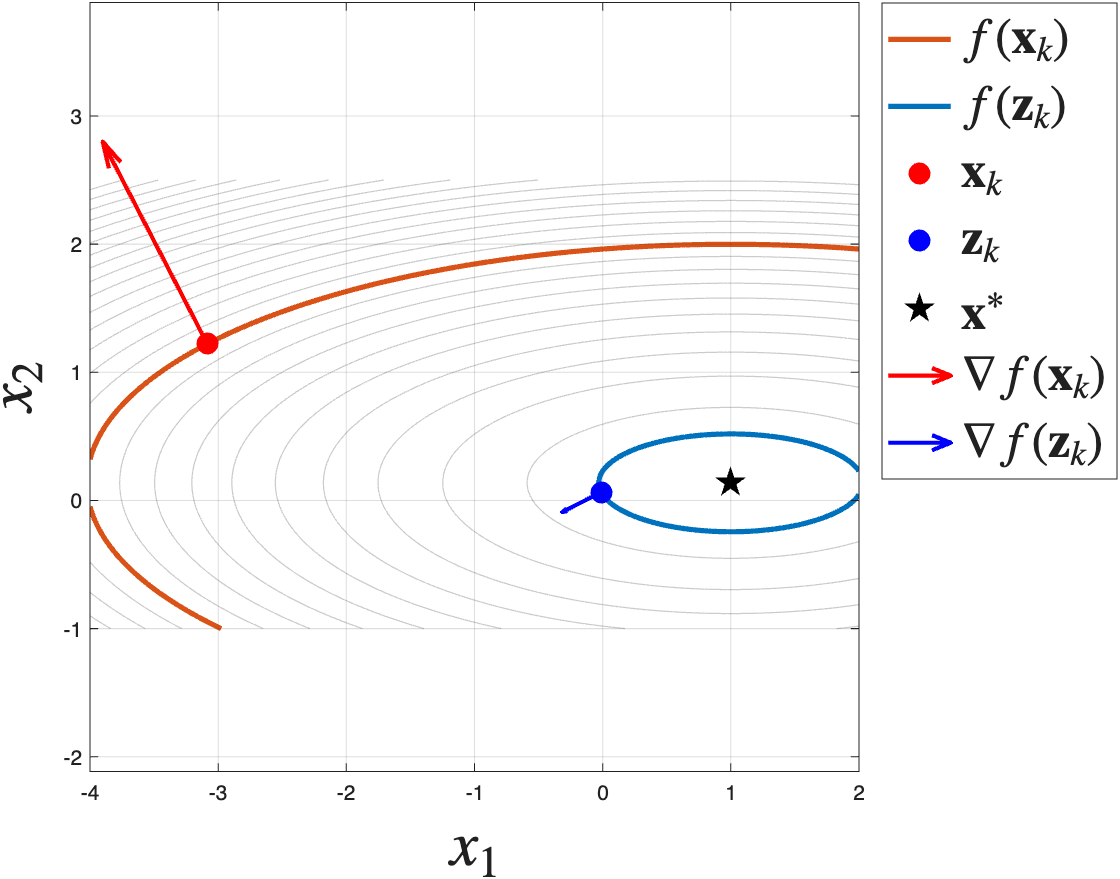}
    \caption{Left: both $\wt\alpha_k$ and $\wt\beta_k$ are positive. 
    Right: $\wt\alpha_k \le 0$ and $\wt\beta_k \ge 0$.}
    \label{fig:cases}
\end{figure}

Notice that, because of \eqref{beta}, at least one of $\wt\alpha_k$ and $\wt\beta_k$ will be positive, and if, in addition, is $f(\bx_k)=f(\bz_k)$, both will be positive and inversely proportional to the norms of the gradients:
\[
\begin{array}{rll}
\wt\alpha_k & = \tfrac{(\bx_k-\bz_k)^\top A \, (\bx_k-\bz_k) }{2 \, \|\nabla f(\bx_k) \|} & = \tfrac{(\bx_k-\bz_k)^\top A \, (\nabla f(\bx_k)-\nabla f(\bz_k)) }{2 \, \|\nabla f(\bx_k) \|}, \\[2mm]
\wt\beta_k & = \tfrac{(\bx_k-\bz_k)^\top A \, (\bx_k-\bz_k) }{2 \, \|\nabla f(\bz_k) \|} & = \tfrac{(\bx_k-\bz_k)^\top A \, (\nabla f(\bx_k)-\nabla f(\bz_k)) }{2 \, \|\nabla f(\bz_k) \|}.
\end{array}
\]

\section{A hybrid method}
\label{sec:algo}

The analysis of Section~\ref{subsec:global_convergence} gives a precise picture of when the Twin coupling is effective, and this picture is the main motivation for the hybrid framework proposed in this section.
\begin{itemize}
\item The Twin step always reduces the mutual distance, and for problem~\eqref{eq:quadratic_problem} it does so at least linearly, so that $\bd_k \to \zero$.
\item As $\bd_k \to \zero$, either the gradients vanish fast enough or the search directions inevitably become parallel: by Proposition~\ref{prop:parallel}, if the gradients stay bounded away from zero, then $1-\gamma_k = \calo(\|\bd_k\|^2)$ and $R_k \to \infty$. In this regime the bound of Lemma~\ref{lem:stepbound} degenerates, the Twin steps may become very large, and the method may diverge, as observed in the undamped run of Table~\ref{tab:comparison}.
\item Damping is not a foolproof remedy. A damping with $\eta_k\to1$ (Propositions~\ref{thm:conv_damped1} and~\ref{thm:conv_damped2}) slows down the alignment but still requires Assumption~\ref{ass:Rk}, which cannot be guaranteed a priori. A damping designed to enforce Assumption~\ref{ass:gamma_bound} either prevents convergence (Remark~\ref{rem:open_loop}) or, in its closed-loop form (Proposition~\ref{prop:closed_loop}), restricts the effective steps to $\calo(1/L)$ (Remark~\ref{rem:lipschitz_limit}).
\item Conversely, as long as $R_k$ remains bounded the Twin iteration is provably well behaved: the stepsizes are bounded and, since $\bd_k \to \zero$, the gradients at both iterates vanish at rate $\calo(\|\bd_k\|^{1/2})$ (Corollary~\ref{cor:gradients}).
\end{itemize}
These facts suggest a natural strategy: exploit the Twin coupling as long as it is effective, monitoring the two quantities that control $R_k = \|\bd_k\|/(1-|\gamma_k|)$, namely the angle between the search directions and the contraction of the mutual distance, and switch to a single-process spectral method as soon as the coupling deteriorates. With this design, Assumption~\ref{ass:Rk} is no longer a hypothesis but is enforced by the algorithm itself during the Twin phase: if the method never leaves the Twin phase, the gradients converge to zero R-linearly (Proposition~\ref{prop:hybrid_twin_phase}), and otherwise the asymptotic behaviour is governed by the spectral method, whose global convergence is well established. This is also consistent with our numerical experience, which shows that the Twin method is most effective when the iterates are far from the solution, and that it creates favourable initial conditions for the subsequent spectral iterations.

The framework grafts a speedup scheme onto ABB${\min}$, analogous to the approach in \cite{DeMagistrisCOAP}. The key distinction is that, whereas their method continuously alternates between acceleration and ABB$_{\min}$, ours switches to ABB$_{\min}$ when parallelism arises between $\bp_k$ and $\bq_k$. Although the framework can in principle be combined with any step-length rule, we focus on ABB$_{\min}$, which is widely used and has proven particularly effective in numerical experiments \cite{Frassoldati:2008}.

\subsection{Twin-ABB$_{\min}$ for the quadratic case }
To keep the search directions separated without relying on a damping sequence, and to leverage the efficiency of spectral methods, we propose the Twin-ABB$_{\min}$ framework with a \textit{restart} mechanism. In other words, the algorithm begins with a Twin phase, applying a restart mechanism whenever the two processes become too close or their gradients approach collinearity. If the restart fails to get out from those critical situations, the method switches to ABB$_{\min}$.

The algorithm starts with the Twin method to exploit geometric cooperation during the early phase. At each iteration, we check the quality of the coupling through two distinct metrics.

\begin{enumerate}
\label{des:mec}
\item 
Once a threshold value for $\gamma^*$ is set, the angle between the search processes is evaluated to ensure that the two processes are not parallel.
If $|\gamma_k| > \gamma^*$, the directions are considered parallel.

\item We check the contraction rate of the mutual distance $\delta_k = \|\bd_k\|$. We define the contraction ratio as $\rho_k = \delta_k / \delta_{k-1}$. If $\rho_k > \rhoc$, it indicates that the mutual attraction has stagnated, and the dual interaction mechanism is no longer driving significant convergence.
\end{enumerate}
If any of these conditions occurs, the algorithm does not immediately terminate the dual process. $\bx_k$ is set equal to the iteration that yields the lowest objective function value up to that point, whilst $\bz_k$ is computed as $\bz_k=\bx_k-\alpha_k^{\text{BB1}}\nabla f (\bx_k)$. The aim of the \textit{restart} is to attempt to restore favorable conditions for the algorithm — in particular, to escape any possible collinearity between the search directions. The number of restarts is controlled by two complementary rules. First, a restart is considered successful only if it allows progress for at least one iteration: if one of the conditions is violated again immediately after a restart, further corrections would lead to a cycle of ineffective updates. Second, at most $N_{\rm r}^{\max}$ restarts are allowed. When either rule is triggered, the geometric acceleration of the Twin strategy is considered
exhausted, and the method switches to ABB$_{\min}$~\cite{Frassoldati:2008}.
The full procedure is reported in Algorithm~\ref{alg:TwinABB}.

\begin{algorithm}[htb!]
\caption{A hybrid Twin-ABB$_{\min}$ method with restart}
\label{alg:TwinABB}
\footnotesize
\begin{algorithmic}[1]
   \State \textbf{Initialization:} $\bx_0, \bz_0 \in \R^n$; ${\sf tol} > 0$; ${\sf maxit} \in \mathbb{N}$; $\gamma^*,\rhoc \in (0,1)$; {$N_{\rm r}^{\max} \in \mathbb{N}$; $\eta_k \in (0,1]$;}
    \State $\delta_0 = \|\bx_0 - \bz_0\|$; $\text{switched} = {\sf false}$
    \State $k^{\text{r}} \leftarrow -1$; \quad $n_{\rm r} \leftarrow 0$
    \For{$k = 0, \dots, {\sf maxit}$}
        \State Compute $\bp_k$ and $\bq_k$ using \eqref{eq:normalized}
        \State $\delta_k = \|\bx_k - \bz_k\|$; \quad $\gamma_k = \bp_k^{\top} \,  \bq_k$; \quad $\rho_k = \delta_k \, / \, \delta_{k-1}$  

        \If{ $(k > 1)$ \textbf{and} $(\rho_k > \rhoc \textbf{ or } {|\gamma_k|} > \gamma^*)$}  \Comment{Check  Conditions}
       \If{ $(k \neq k^{\text{r}}+1)$ \textbf{and} $(n_{\rm r} < N_{\rm r}^{\max})$ }
    \State $\bx_k = \text{best}(\bx_k, \bz_k), \quad \bz_k = \bx_k - \alpha_k^{\text{BB1}} \, \nabla f(\bx_k)$
    \State $k^{\text{r}} = k$; \quad $n_{\rm r} = n_{\rm r} + 1$
    \State {\bf continue}
\Else \Comment{Switch to ABB$_{\min}$}
            \State $\text{switched} = {\sf true}$; \ \textbf{break} 
        \EndIf
    \EndIf
             \State Compute $(\alpha_k, \beta_k)$ from~\eqref{eq:final_step_selection}
            \State $\bx_{k+1} = \bx_k +  \, {\eta_k} \alpha_k \, \bp_k$, \quad $\bz_{k+1} = \bz_k + \, {\eta_k} \beta_k \, \bq_k$   
            \State {\bf if} $\min \{\|\nabla f(\bx_k)\|, \, \|\nabla f(\bz_k)\| \} \le {\sf tol} \cdot \|\nabla f(\bx_0)\| $, {\bf break}; \ {\bf end if}
        \EndFor
        
        \If{ $\text{switched}$ } \Comment{Stage 2: ABB$_{\min}$ Phase}
            \State $\bx_k = \frac12 (\bx_k + \bz_k)$
            \For{ $j = k+1, \dots, {\sf maxit}$ }
                \State Compute $\beta_j^{\text{ABB}}$ using ABB$_{\min}$ 
                \State $\bx_{j} = \bx_{j-1} - \beta_{j-1}^{\text{ABB}} \, \nabla f(\bx_{j-1})$
                \State {\bf if} {\tt stopping criteria} , {\bf break}; \ {\bf end if}
            \EndFor
        \EndIf
    \end{algorithmic}
\end{algorithm}

Specifically, for the {\tt stopping criteria} in line~25 we take
\[\|\nabla f(\bx_j)\| \le 10^{-6} \cdot \|\nabla f(\bx_0)\| \quad \text{or} \quad |f(\bx_j)-f(\bx_{j-1})| \le 10^{-9} \cdot \, |f(\bx_j)|.\]

\begin{remark} \rm
In the hybrid Twin-ABB$_{\min}$ framework the damping factors $\eta_k$ are optional: situations in which the Twin coupling deteriorates are handled by the restart and switching mechanism, and the convergence of the Twin phase (Proposition~\ref{prop:hybrid_twin_phase}) holds for any $\eta_k\in(0,1]$. A damping with $\eta_k<1$ is, however, always required for two-dimensional problems ($n=2$); otherwise, the exact Twin step with $\eta_k=1$ would typically lead to an immediate intersection of the search lines, causing the Twin process to stop after a single iteration. For $n>2$ we consider both variants in Section~\ref{sec:experiments}, with the non-decreasing sequence $\eta_k = 0.9 + 0.1\,\frac{k-1}{k+1}\to1$ in the damped case, which falls within the setting of Propositions~\ref{thm:conv_damped1} and~\ref{thm:conv_damped2}.
\label{remark5.1}
\end{remark}

\subsection{Twin-ABB$_{\min}$ for general functions}
\label{sec:nonquadcase}

Building upon the quadratic framework, we now extend the practical implementation of Twin-ABB$_{\min}$ to general unconstrained optimization problems. The two-phase architecture and the stability monitoring mechanism are identical to those described in Section~\ref{des:mec}. 
The main difference lies in the selection of the stepsize. Since the optimal steps derived from the unconstrained Twin-step system \eqref{eq:system2x2_normalized} do not guarantee a sufficient decrease for nonquadratic functions, they are instead used as initial trial steps ($\alpha_k^{\text{Twin}}, \beta_k^{\text{Twin}}$) for a nonmonotone line search. This strategy, highly effective for spectral gradient methods \cite{Grippo:1986}, enforces stable descent. Specifically, we search for $\alpha_k = \alpha_k^{\text{Twin}} \upsilon^j$ (with $\upsilon \in (0,1)$ and $j=0, 1, \dots$) satisfying the nonmonotone Armijo condition:
\begin{equation}
\label{eq:nonmonotone_cond}
f(\bx_k + \alpha_k \, \bp_k) \le \max_{0 \, \le \, i \, \le \, \min(k, M)} f(\bx_{k-i}) + \nu \, \alpha_k \, \nabla f(\bx_k)^{\top} \, \bp_k,
\end{equation}
where $\nu \in (0, 1)$ is a small constant and $M$ is a nonnegative integer determining the memory length. The same procedure is applied independently to determine the step $\beta_k$ for the auxiliary process $\bz_k$, using $\beta_k^{\text{Twin}}$ as the starting guess. The use of the Twin solution as the initial trial step allows the algorithm to capture the local curvature information shared between the two processes, often reducing the number of backtracking operations required by the line search. While the Twin strategy is effective in the early stages of optimization, particularly for navigating narrow valleys, maintaining two coupled processes becomes computationally redundant as the iterates converge to the solution or when the search directions become collinear. The complete procedure is summarized in Algorithm~\ref{alg:Twin_nonquad}.

\begin{algorithm}[ht]
\caption{A hybrid Twin-ABB$_{\min}$ method for general functions}
\label{alg:Twin_nonquad}
\footnotesize
\begin{algorithmic}[1]
    \State \textbf{Initialization:} $\bx_0, \bz_0 \in \R^n$; ${\sf tol} > 0$; ${\sf tol_f} > 0$; ${\sf maxit}, M \in \mathbb{N}$; $\gamma^*, \rhoc \in (0,1)$; {$N_{\rm r}^{\max} = 10$; $\eta_k \in (0.1];$}
    \State $\delta_0 = \|\bx_0 - \bz_0\|$; $\text{switched} = {\sf false}$
    \State $k^{\text{r}} \leftarrow -1${; \quad $n_{\rm r} \leftarrow 0$}
    \For{$k = 0, \dots, {\sf maxit}$}
        \State Compute $\bp_k$ and $\bq_k$ using \eqref{eq:normalized}
        \State $\delta_k = \|\bx_k - \bz_k\|$; \quad $\gamma_k = \bp_k^{\top} \bq_k$; \quad $\rho_k = \delta_k \, / \, \delta_{k-1}$
        \If{$k > 1$ \textbf{and} $(\rho_k > \rhoc$ \textbf{or} ${|\gamma_k|} > \gamma^*)$} \label{alg3:check}
            \Comment{Check conditions}
            \If{$k \neq k^{\text{r}}+1$ {\textbf{and} $n_{\rm r} < N_{\rm r}^{\max}$}}
                \State $\bx_k = \text{best}(\bx_k, \bz_k), \quad {\bz_k = \bx_k - \alpha_k^{\text{BB1}} \, \nabla f(\bx_k)}$
                \Comment{Keep best point w.r.t.~$f$-value and reset $\bz$ via BB1 step}
                \State $k^{\text{r}} = k${; \quad $n_{\rm r} = n_{\rm r} + 1$}
                \State {\bf continue}
            \Else
                \State switched = {\sf true}; \ \textbf{break} \Comment{Switch to ABB$_{\min}$}
            \EndIf
        \EndIf
        \State Compute $(\alpha_k^{\text{Twin}}, \beta_k^{\text{Twin}})$ from~\eqref{eq:final_step_selection}
        \State Find $\alpha_k, \, \beta_k$ starting from $\alpha_k^{\text{Twin}}, \, \beta_k^{\text{Twin}}$ satisfying \eqref{eq:nonmonotone_cond}
        \State $\bx_{k+1} = \bx_k + {\eta_k} \alpha_k \, \bp_k$; \quad $\bz_{k+1} = \bz_k + {\eta_k} \beta_k \, \bq_k$
        \State {\bf if} $\min \{\|\nabla f(\bx_k)\|, \, \|\nabla f(\bz_k)\| \} \le {\sf tol} \cdot \|\nabla f(\bx_0)\|$, {\bf break}; \ {\bf end if}
    \EndFor
    \If{switched} \Comment{Phase 2: ABB$_{\min}$ phase}
        \State $\bx_k = \tfrac12 \, (\bx_k + \bz_k)$
        \For{$j = k+1, \dots, {\sf maxit}$}
            \State Compute stepsize $\beta_j^{\text{ABB}}$ using ABB$_{\min}$
            \State Find $\beta_j$ starting from $\beta_j^{\text{ABB}}$ satisfying \eqref{eq:nonmonotone_cond}
            \State $\bx_{j} = \bx_{j-1} - \beta_{j-1} \, \nabla f(\bx_{j-1})$
            \State {\bf if} {\tt stopping criteria}, {\bf break}; \ {\bf end if} \label{alg3:stop}
        \EndFor
    \EndIf
\end{algorithmic}
\end{algorithm}

Specifically, for the {\tt stopping criteria} in line~29 is the same of the quadratic case. The global convergence of the proposed hybrid method relies on its structural design. Since the Twin phase limits the number of restarts, the algorithm either converges within the Twin phase {(Proposition~\ref{prop:hybrid_twin_phase} below)} or executes it for a finite number of iterations before switching to the ABB$_{\min}$ phase. Therefore, the asymptotic behavior is entirely governed by the single-process phase. Suppose that $f$ is bounded below in $\R^n$ and that $f$ is continuously differentiable in an open set $\mathcal{N}$ containing the level set $\mathcal{L} = \{\bx: f(\bx) \le f(\bx_k)\}$, where $\bx_k$ is the initial point of the ABB$_{\min}$ phase. Assume also that the gradient $\bg = \nabla f$ is Lipschitz continuous on $\mathcal{N}$.  
In the second phase, the search direction is the steepest descent direction $-\nabla f(\bx_k)$, meaning the angle with the negative gradient is exactly zero, naturally satisfying the Zoutendijk condition. The stepsizes are selected via the nonmonotone backtracking procedure \eqref{eq:nonmonotone_cond}. Consequently, applying the standard convergence theory for nonmonotone line search methods (see \cite{Grippo:1986} and \cite[Thm.~3.2]{NocedalWright2006}), it holds that the hybrid Twin-ABB$_{\min}$ algorithm globally converges, i.e. 
$
\lim_{k \to \infty} \nabla f(\bx_k) = \textbf{0}.
$ 
Therefore, the Twin-ABBmin framework preserves the guarantees of global convergence of
spectral methods, while effectively exploiting double-path geometric acceleration during the
critical initial stages of optimization. The following result shows that the Twin phase is itself convergent whenever the method never leaves it: the switching test enforces Assumption~\ref{ass:Rk}, and even $R_k \to 0$, by construction.

\begin{proposition}
\label{prop:hybrid_twin_phase}
Assume that $\nabla f$ is Lipschitz continuous with constant $L>0$, and consider the Twin phase of Algorithm~\ref{alg:TwinABB} or Algorithm~\ref{alg:Twin_nonquad}. Suppose that the method never switches to ABB$_{\min}$, and let $k_0$ be the iteration of the last restart ($k_0 = 1$ if no restart occurs). Then, for all $k > k_0$,
\[
R_k \le \frac{\|\bd_k\|}{1-\gamma^*}, \qquad \|\bd_k\| \le \rhoc^{\,k-k_0}\,\|\bd_{k_0}\|, \qquad
\|\nabla f(\bx_k)\|, \ \|\nabla f(\bz_k)\| \le L\sqrt{\frac{2}{1-\gamma^*}}\;\|\bd_k\|.
\]
In particular, Assumption~\ref{ass:Rk} holds, $R_k \to 0$, and $\nabla f(\bx_k), \nabla f(\bz_k) \to \zero$ R-linearly, with rate $\rhoc$. For problem~\eqref{eq:quadratic_problem}, $\bx_k \to \bx^*$ and $\bz_k \to \bx^*$ R-linearly, with $\|\bx_k - \bx^*\|, \|\bz_k - \bx^*\| = \calo(\rhoc^{\,k})$.
\end{proposition}
\begin{proof}
\hspace*{0.2cm}
After the last restart the conditions in line~7 are never satisfied, since otherwise the method would either restart again or switch to ABB$_{\min}$; hence $|\gamma_k| \le \gamma^*$ and $\rho_k = \|\bd_k\|/\|\bd_{k-1}\| \le \rhoc$ for all $k > k_0$. The first two bounds follow immediately. By Proposition~\ref{prop:key}, which holds for any pair of points and therefore regardless of how the stepsizes are computed,
\[
\|\nabla f(\bx_k)\|^2 \le L^2\|\bd_k\|\max\{2R_k, \|\bd_k\|\} \le L^2\|\bd_k\|^2\max\Big\{\frac{2}{1-\gamma^*}, 1\Big\} = \frac{2L^2}{1-\gamma^*}\,\|\bd_k\|^2,
\]
and similarly for $\bz_k$. For problem~\eqref{eq:quadratic_problem}, use $\lambda_1\|\bx - \bx^*\| \le \|\nabla f(\bx)\|$.
\end{proof}
Note that no assumption on $R_k$ is needed: the switching test itself guarantees that $|\gamma_k| \le \gamma^*$ and $\rho_k \le \rhoc$ as long as the Twin phase is not abandoned. Thus Twin-ABB$_{\min}$ converges in all cases: either it remains in the Twin phase, where the gradients vanish R-linearly by Proposition~\ref{prop:hybrid_twin_phase}, or it switches to ABB$_{\min}$, whose global convergence follows from the nonmonotone line search theory recalled above.

\section{Numerical experiments}
\label{sec:experiments}
To assess the efficiency of the proposed Twin-ABB$_{\min}$ strategy, we compare it with ABB$_{\min}$ method \cite{Frassoldati:2008}.
Since our algorithm falls back to ABB$_{\min}$ once the Twin coupling is no longer beneficial, the comparison specifically highlights the acceleration provided by the initial Twin phase. We use the performance profile technique introduced by Dolan and Moré \cite{DolanMore2002}, comparing the algorithms based on the number of function evaluations or number of gradient evaluations required to satisfy the stopping criteria. 

Since we want to avoid starting with parallel search directions, selecting an initial point $\bz_0$ that enforces orthogonality of $\bp_0$ and $\bq_0$ is a quite natural choice.
While this strategy seems to works very well in our experiments, we emphasize that identifying alternative good choices for $\bz_0$ remains an open topic for future research.
The details of our orthogonal initialization are explained in the next subsection.

\subsection{Selection of $\bz_0$ given $\bx_0$} \label{sec:z0}
As we will see, the Twin method can yield very promising computational results, for both quadratic and general functions. Nevertheless, we wish to emphasise that the appropriate choice of the auxiliary starting point $\bz_0$ is crucial for an efficient implementation. In particular, this choice should favour the dual process, so as to achieve a significant reduction in the objective function. The analysis presented so far shows that collinearity between the directions should be avoided whenever possible. It therefore seems natural to choose $\bz_0$ such that
\begin{equation}
\label{orthogonality}
\nabla f(\bx_0)^\top\nabla f(\bz_0) = 0.
\end{equation}
We begin with the quadratic case. 
First, we select a random direction $\bv \in \R^n$ drawn from a standard normal distribution. Then, we compute the stepsize $\theta$ along $\bv$ such that the gradient at  $\bz_0 = \theta \bv$ is orthogonal to the initial gradient $\nabla f(\bx_0)$. By imposing the orthogonality condition $\nabla f(\theta \, \bv)^{\top} \, \nabla f(\bx_0) = 0$, we obtain
\[
\left(\theta \, A \bv - \bb \right)^{\top} (A\bx_0-\bb) = \nabla f(\bx_0)^{\top} \left( \theta \, A \bv - \bb \right) = 0.
\]
Solving for $\theta$, we derive the analytic expression for the orthogonal step:
\[
\theta = \frac{\bp_0^{\top} \, \bb}{\bp_0^{\top} A \bv}.
\]
In general, we follow a procedure that is entirely analogous to the one used in the quadratic case,
bearing in mind that condition (29) obviously cannot be satisfied exactly. In this case as well,
we start with the ansatz $\bz_0 = \theta \bv$.
For the general nonlinear case, where $f \in C^2$, we approximate the orthogonality condition $\nabla f(\bx_0)^\top \nabla f(\bz_0) = 0$ using a first-order Taylor expansion:
$
\nabla f(\bz_0) \approx \nabla f(\bx_0) + H_0\,(\bz_0 - \bx_0) =  \theta \, H_0\bv - H_0\bx_0 +\nabla f(\bx_0) ,
$
where $H_0$ is the Hessian matrix of $f$ at $\bx_0$.
Substituting this into the orthogonality condition and neglecting higher-order terms gives
\[
\nabla f(\bx_0)^\top \bigl( \nabla f(\bx_0) + \theta \, H_0\bv - H_0\,\bx_0 \bigr) = 0,
\]
which yields
\[
\theta = \frac{\nabla f(\bx_0)^\top [H_0\bx_0- \nabla f(\bx_0)]}{\nabla f(\bx_0)^\top H_0\bv}.
\]
For the quantities $H_0\bv$ and $H_0\bx_0$  a secant condition based estimation can be computed:
\[ H_0\bx_0= H_0(\bx_0-\mathbf{0}) \approx \nabla f(\bx_0)-\nabla f(\mathbf{0}),
\qquad
H_0\bv= H_0((\bx_0+\bv)-\bx_0) \approx \nabla f(\bx_0+\bv)-\nabla f(\bx_{0})\]
which leads to the following expression for $\theta$
\[
\theta= - \frac{\nabla f(\bx_0)^\top\nabla f(\mathbf{0})}{\nabla f(\bx_0)^\top(\nabla f(\bx_0+\bv)-\nabla f(\bx_{0}))}.
\]
This approximation does not require a finite-difference parameter and uses historical gradient information. Now that the initialization for the auxiliary process has been established, the next section will formalize the complete hybrid methodology, building on these practical setups.

\subsection{Quadratic case}
The test suite consists of a set of strictly convex problems of the form \eqref{eq:quadratic_problem}, randomly generated using MATLAB. To ensure a challenging test environment, we varied the problem dimension $n \in \{1000, 5000, 10000\}$ and the condition number $\kappa(A) \in \{10^4, 10^5, 10^6, 10^7\}$. 
The experiments were performed in the MATLAB R2025b environment on a 14-inch MacBook Pro equipped with an Apple M3 Pro chip and 18 GB of RAM, running on macOS Sequoia (Version 15.6.1).
To evaluate performance of the methods against different curvature geometries, following \cite{di_serafino2018, DeMagistrisCOAP, Asmudis:2013}, we designed the hessian $A$ to exhibit three distinct spectral distributions.
We consider three eigenvalue distributions for the Hessian $A$. 
\textit{Bimodal spectrum}: eigenvalues are clustered at the extremes of $[1, \kappa]$, with half uniformly sampled in $[1, \, 0.2\kappa]$ and the other half in $[0.8\kappa, \, \kappa]$. 
\textit{Logarithmic distribution}: eigenvalues are geometrically spaced within the interval. 
\textit{Linear distribution}: eigenvalues progress uniformly across the spectrum. 

For the linear term and the solution, we implement two distinct generation strategies to increase dataset variability. 
In \textit{Strategy A}, the exact minimizer $\bx^*$ is predetermined with components randomly drawn from $[-5, 5]$ using three distributions: uniform, normal, and sparse normal (density $0.4$). The vector $\bb$ is then computed as $\bb = A \bx^*$. In \textit{Strategy B}, we set the right-hand side vector $\bb$ to a vector of all ones, implicitly defining the solution through the linear system $A\bx = \bb$. The starting point $\bx_0$ is generated uniformly at random with components drawn from $[-5, 5]$. We perform 5 runs for each problem configuration with different random seeds, resulting in a total of 720 test instances. The stopping criteria is $\|\nabla f(\bx_k)\| \le 10^{-7} \cdot \|\nabla f(\bx_0)\|$, with a maximum of $8n$ iterations. The internal parameters for Twin-ABB$_{\min}$ are set to $\rhoc = 0.9$, $\gamma^* = 0.9$, 
$N_{\rm r}^{\max} = 10$ and $\eta_k = 0.9 + 0.1\,\tfrac{k-1}{k+1}$, while the ABB$_{\min}$ parameters follow the configuration suggested in \cite{Frassoldati:2008}.

\begin{figure}[htbp]
    \centering

\includegraphics[width=0.32\textwidth]{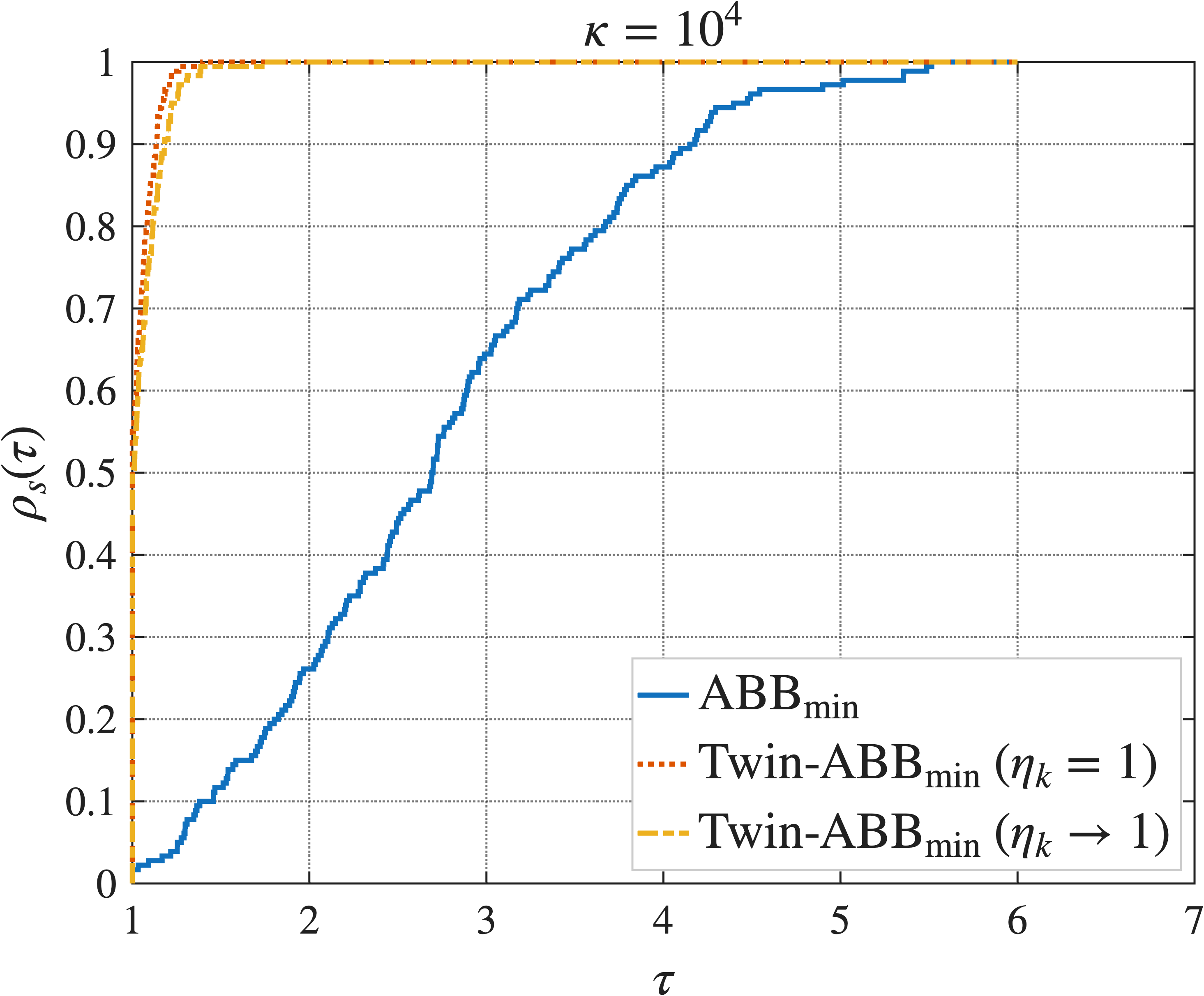}
\includegraphics[width=0.32\textwidth]{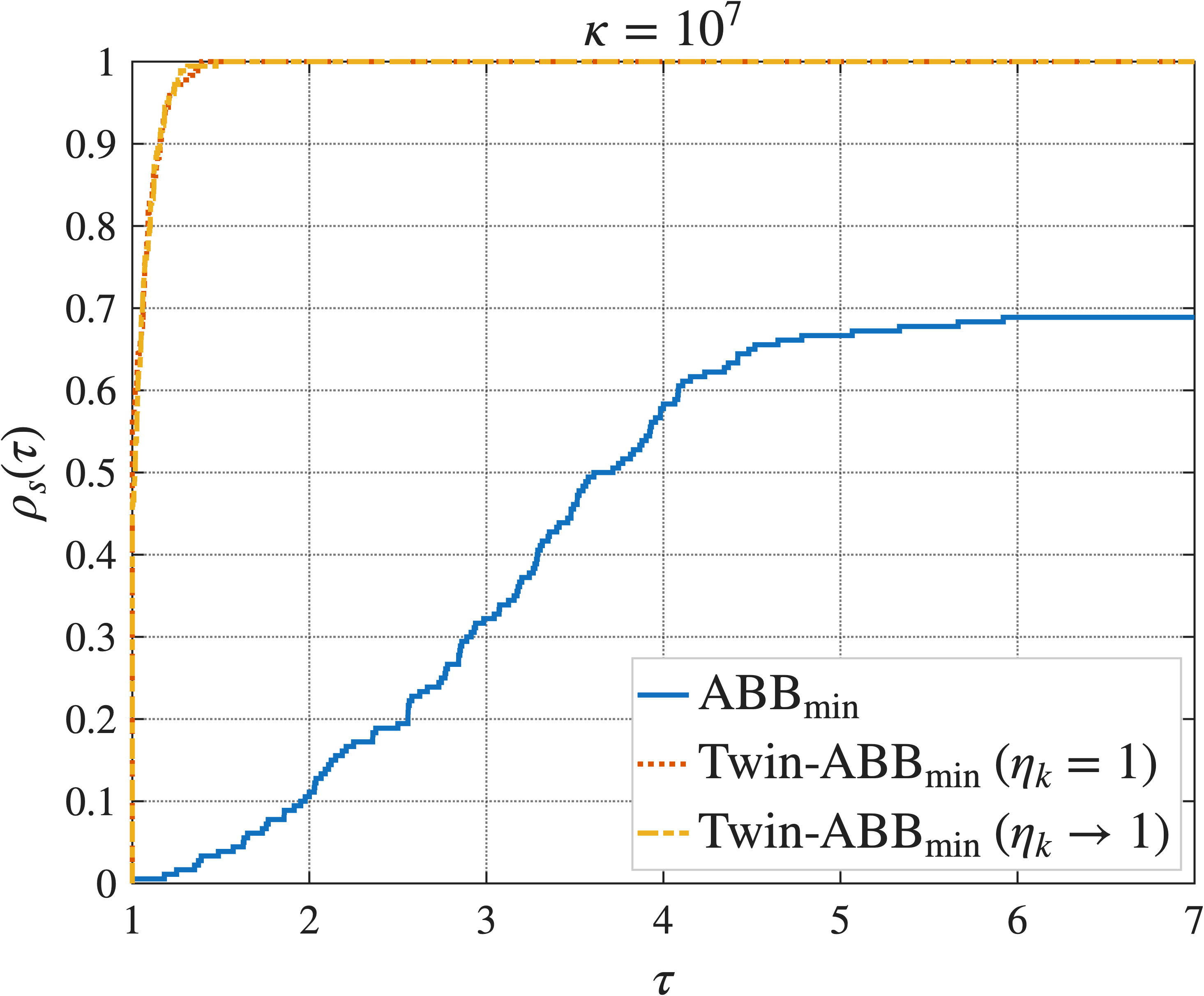}
\includegraphics[width=0.32\textwidth]{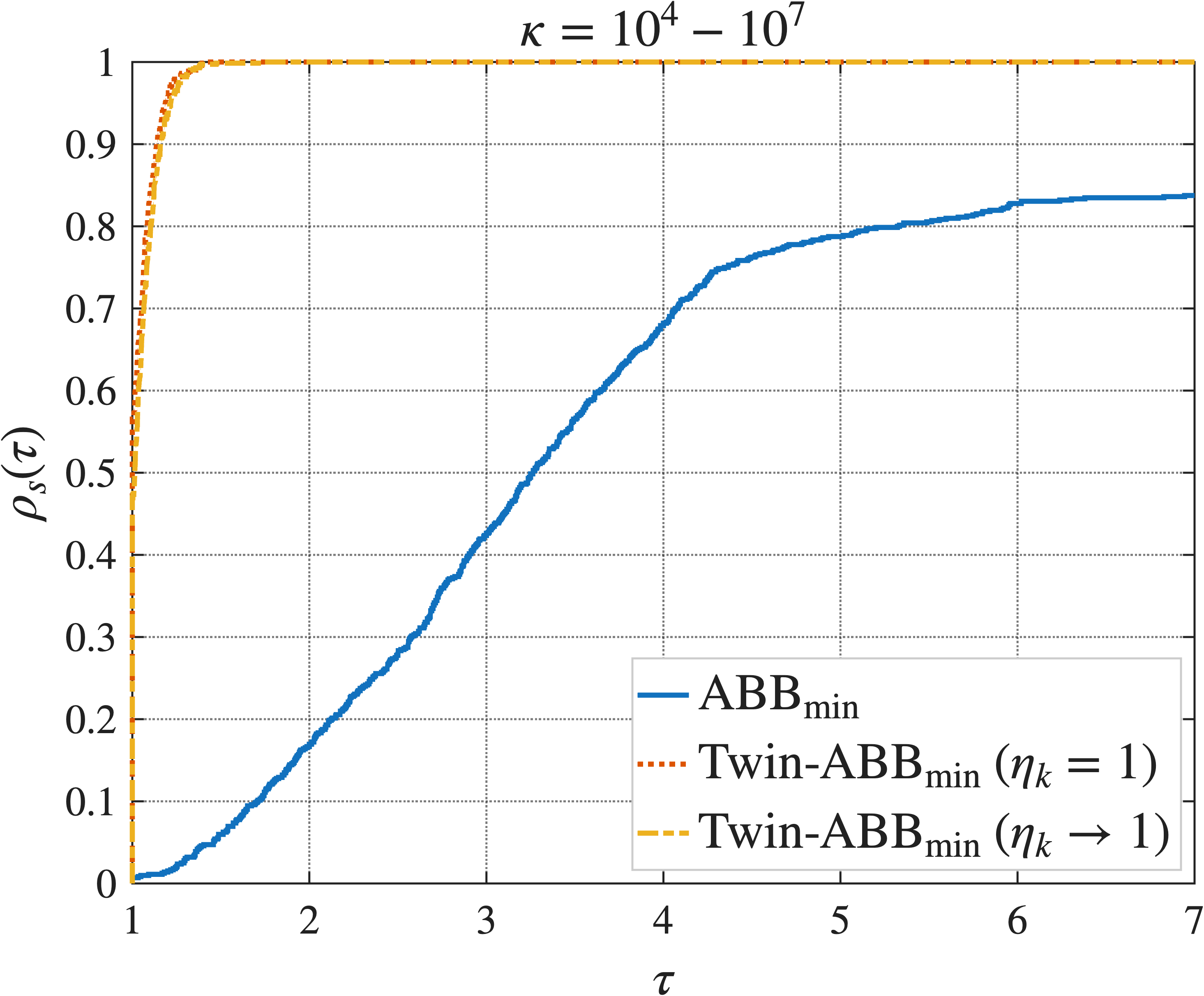}

      \caption{ Performance profile of $\rm ABB_{\min}$ and Twin-ABB$_{\min}$ on $720$ problems, in terms of number of gradient evaluations.}
    \label{fig:perf_profile}
\end{figure}

\begin{figure}[htbp]
    \centering 
 \includegraphics[width=0.49\textwidth]{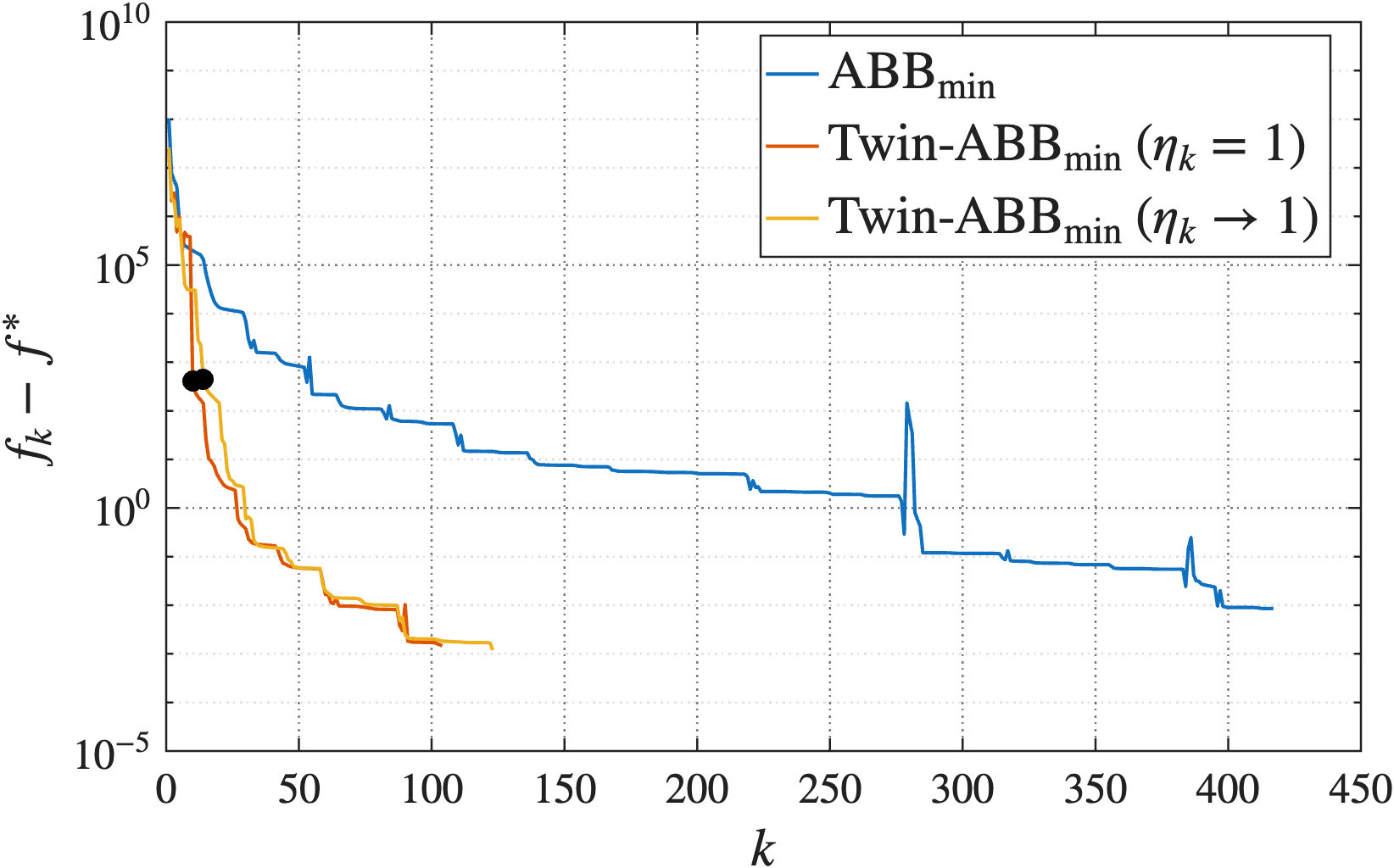}
  \hfill
\includegraphics[width=0.49\textwidth]{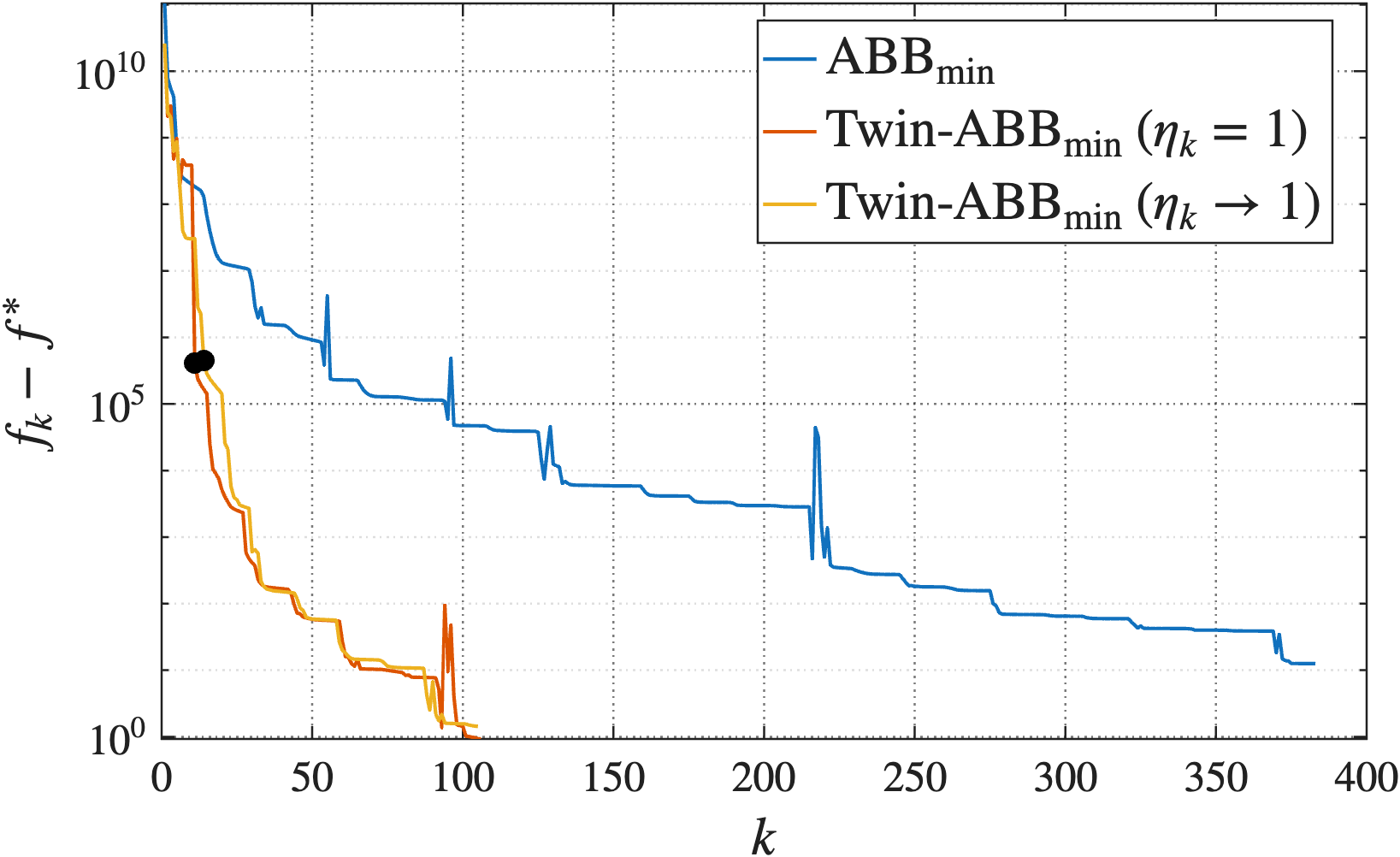}

    \caption{$\rm ABB_{\min}$ VS Twin-ABB$_{\min}$ on a problem with $n=1000$ and $\kappa(A)=10^4$ for the figure on the left, and $\kappa(A)=10^7$ for the figure on the right. The black dot represents the iterates in witch the method switches to ABB$_{\min}$, ($f_k=\min \{f(\bx_k),f(\bz_k)\})$.}
    \label{fig:convergence_history}
\end{figure}

The computational results are summarized in Figure~\ref{fig:perf_profile}.  Note that all the comparisons were made in terms  of  gradient evaluation. A comparison in terms of iterations would have been unfair, since a Twin iteration is twice more expensive compared to ABB$_{\min}$. It can be observed that the Twin steps  significantly improve the performance of Algorithm ABB$_{\min}$, with a particularly marked difference when considering problems with higher conditioning. 
 Furthermore, the performance curve of Twin-ABB$_{\min}$ remains above that of ABB$_{\min}$ for all values of the performance ratio $\tau$.
The effect of the Twin steps in the two phases (Twin and ABB$_{\min}$) is clearly illustrated in Figure~\ref{fig:convergence_history}, which shows the convergence history of the two algorithms on two problems with $10^4$ variables, with condition numbers of $10^4$ and $10^7$, respectively. 
The first phase creates a favorable warm-up for the second phase, from which the ABB$_{\min}$ algorithm derives significant benefits.

\subsection{General case}
To evaluate performance of the proposed algorithm on general nonlinear landscapes, we considered a comprehensive test suite selected from the CUTEst library. Specifically, we utilized the OPM collection described by Gratton and Toint \cite{GrattonToint2021}, along with additional problems collected by Andrei \cite{Andrei2008}. The OPM collection provides a direct MATLAB interface for  CUTEst problems, allowing for seamless integration without external Fortran compilation.
The test set consists of $164$ unconstrained optimization problems characterized by different properties, including ill-conditioning, nonconvexity, and variable dimensions ranging from $n=2$ to $n=10000$. The specific dimensions for each problem are selected to ensure a balanced mix of small-scale, medium-scale, and large-scale scenarios. Table~\ref{tab:cutest_problems} lists the complete set of problems used in the experimentation, detailing their ID, name, and the dimension $n$ adopted.
\begin{table}[htb!]
\centering
\scalebox{0.6}{
\begin{tabular}{rlr|rlr|rlr|rlr} \hline
\textbf{ID} & \textbf{Problem Name} & \textbf{$n$} & \textbf{ID} & \textbf{Problem Name} & \textbf{$n$} & \textbf{ID} & \textbf{Problem Name} & \textbf{$n$} & \textbf{ID} & \textbf{Problem Name} & \textbf{$n$} \\ \hline \rule{0pt}{2.3ex}%
1 & Almost Pert. Quadratic & 1000 & 2 & ARGAUSS & 3 & 3 & ARGLINA & 10 & 4 & ARGLINB & 10 \\
5 & ARGLINC & 10 & 6 & ARGTRIG & 10 & 7 & ARWHEAD & 10 & 8 & BARD & 3 \\
9 & BDARWHD & 100 & 10 & Bdexp & 1000 & 11 & BDQRTIC & 1000 & 12 & BEALE & 2 \\
13 & Biggsb1 & 1000 & 14 & Biggs6 & 6 & 15 & Booth & 2 & 16 & Box3 & 3 \\
17 & Brkmcc & 2 & 18 & Brownal & 2 & 19 & Brownbs & 2 & 20 & Brownden & 4 \\
21 & Broyden Tridiagonal & 10 & 22 & Broydenbd & 10 & 23 & Chandheu & 57 & 24 & Chebyqad & 10 \\
25 & Cliff & 2 & 26 & Clustr & 2 & 27 & Cosine & 10000 & 28 & Crglvy & 10 \\
29 & CUBE & 10 & 30 & Curly10 & 30 & 31 & Curly20 & 30 & 32 & Curly30 & 40 \\
33 & Deconvu & 51 & 34 & Diagonal 1 & 1000 & 35 & Diagonal 2 & 1000 & 36 & Diagonal 3 & 1000 \\
37 & Diagonal 4 & 1000 & 38 & Diagonal 5 & 1000 & 39 & Diagonal 6 & 1000 & 40 & Diagonal 7 & 1000 \\
41 & Diagonal 8 & 1000 & 42 & Diagonal 9 & 1000 & 43 & DIXMAANA & 900 & 44 & DIXMAANB & 900 \\
45 & DIXMAANC & 900 & 46 & DIXMAAND & 900 & 47 & DIXMAANE & 900 & 48 & DIXMAANF & 900 \\
49 & DIXMAANG & 900 & 50 & DIXMAANH & 900 & 51 & DIXMAANI & 900 & 52 & DIXMAANJ & 900 \\
53 & DIXMAANK & 900 & 54 & DIXMAANL & 900 & 55 & DIXON3DQ & 10 & 56 & DQDRTIC & 1000 \\
57 & Dqrtic & 10 & 58 & EDENSCH & 1000 & 59 & EG2 & 10 & 60 & Eg2s & 10 \\
61 & Eigenals & 110 & 62 & Eigenbls & 110 & 63 & Eigencls & 462 & 64 & ENGVAL1 & 10 \\
65 & ENGVAL2 & 3 & 66 & Expfit & 2 & 67 & Explin1 & 1000 & 68 & Explin2 & 1000 \\
69 & Extended BD1 & 1000 & 70 & Extended Cliff & 1000 & 71 & Extended Hiebert & 1000 & 72 & Ext. Himmelblau & 1000 \\
73 & Extended Maratos & 1000 & 74 & Extended Powell & 1000 & 75 & Extended PSC1 & 1000 & 76 & Ext. Quad. Exp. EP1 & 1000 \\
77 & Ext. Quad. Penalty QP1 & 1000 & 78 & Ext. Quad. Penalty QP2 & 1000 & 79 & Extended TET & 1000 & 80 & Ext. Tridiagonal 1 & 1000 \\
81 & Ext. Tridiagonal 2 & 1000 & 82 & Extended Wood & 1000 & 83 & Extended Beale & 1000 & 84 & Ext. Denschnb & 1000 \\
85 & Ext. Denschnf & 1000 & 86 & Ext. Freud. Roth & 1000 & 87 & Extended Penalty & 1000 & 88 & Ext. Rosenbrock & 1000 \\
89 & Ext. Trigonometric & 1000 & 90 & Ext. White \& Holst & 1000 & 91 & FLETCBV3 & 1000 & 92 & FLETCHCR & 1000 \\
93 & Fminsurf & 1024 & 94 & Full Hessian FH1 & 1000 & 95 & Full Hessian FH2 & 1000 & 96 & Full Hessian FH3 & 1000 \\
97 & Generalized PSC1 & 1000 & 98 & Generalized Quartic & 1000 & 99 & Gen. White \& Holst & 1000 & 100 & Gen. Rosenbrock & 1000 \\
101 & GENHUMPS & 10 & 102 & Gottfr & 2 & 103 & Gulf & 3 & 104 & Hager & 1000 \\
105 & Hairy & 2 & 106 & HARKERP2 & 1000 & 107 & Helix & 10 & 108 & Himmelh & 1000 \\
109 & Himmelbg & 1000 & 110 & INDEF & 1000 & 111 & Integreq & 10 & 112 & Jensmp & 2 \\
113 & Kowosb & 4 & 114 & LIARWHD & 1000 & 115 & Mancino & 10 & 116 & Mexhat & 2 \\
117 & Meyer3 & 3 & 118 & Mccormck & 1000 & 119 & Msqrtals & 16 & 120 & Msqrtbls & 16 \\
121 & Ncb20b & 21 & 122 & Ncb20c & 30 & 123 & NONDIA & 1000 & 124 & NONDQUAR & 1000 \\
125 & NONSCOMP & 1000 & 126 & Nzf1 & 13 & 127 & Osbornea & 5 & 128 & Osborneb & 11 \\
129 & Partial Pert. Quad. & 1000 & 130 & Penalty 1 & 10 & 131 & Penalty 2 & 10 & 132 & Penalty 3 & 10 \\
133 & Pert. Quad. Diagonal & 1000 & 134 & Pert. Trid. Quad. & 1000 & 135 & Perturbed Quadratic & 1000 & 136 & Powellbs & 2 \\
137 & Powellsg & 4 & 138 & Powellsq & 2 & 139 & Power & 10 & 140 & Quadratic QF1 & 1000 \\
141 & Quadratic QF2 & 1000 & 142 & QUARTC & 1000 & 143 & Raydan 1 & 1000 & 144 & Raydan 2 & 1000 \\
145 & Recipe & 3 & 146 & Rosenbr & 10 & 147 & S308 & 2 & 148 & Schmvett & 3 \\
149 & Scurly 10 & 30 & 150 & Scurly 20 & 30 & 151 & SINCOS & 1000 & 152 & Sine & 1000 \\
153 & SINQUAD & 1000 & 154 & Sisser & 2 & 155 & Staircase 1 & 1000 & 156 & Staircase 2 & 1000 \\
157 & TRIDIA & 1000 & 158 & Tridiagonal 1 & 1000 & 159 & Tridiagonal 2 & 1000 & 160 & VARDIM & 1000 \\
161 & Woods & 12 & 162 & Yfitu & 3 & 163 & Zangwil2 & 2 & 164 & Zangwil3 & 3 \\ \hline
\end{tabular}
}
\caption{\footnotesize \rm{Complete list of the 164 test problems from the OPM/CUTEst collection utilized in the experiments.}} 
\label{tab:cutest_problems}
\end{table}

The numerical experiments are conducted using the  starting points $\bx_0$ provided by  CUTEst  for each problem. A maximum limit of 10000 gradient evaluations is imposed. From the initial set listed in Table~\ref{tab:cutest_problems}, we excluded a small (12) subset of instances where numerical overflows (NaN) occurred or where the objective function is undefined in the search region. We also excluded 12 problems in which the 2 algorithms reached different solutions. To provide a comprehensive overview of the comparison between the two methods, we present separate performance profiles for the number of gradient evaluations, the number of function evaluations, and the number of iterations. 

\begin{figure}[htb!]
\centering
\includegraphics[width=0.32\textwidth, height=4cm]{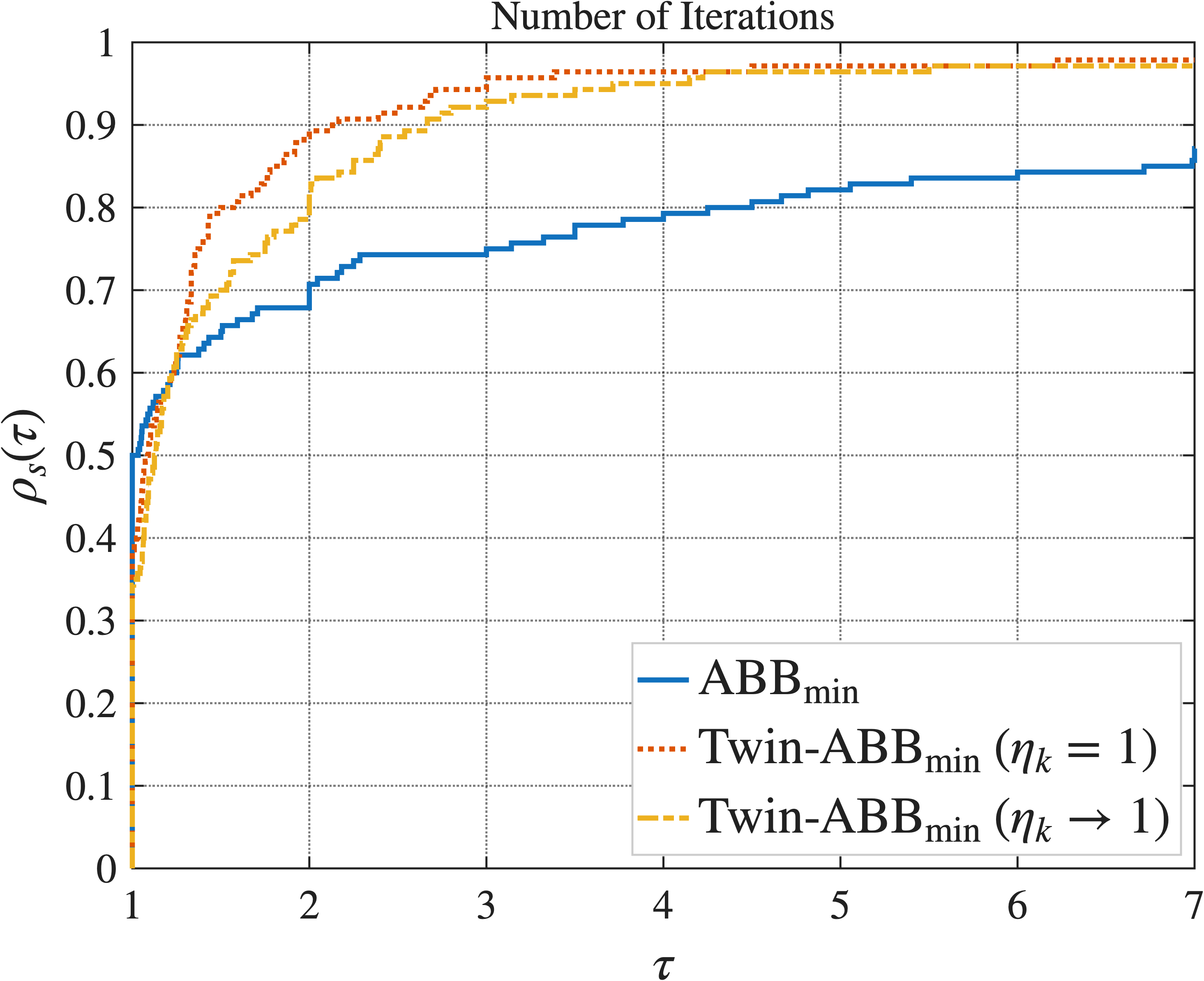} 
  \hfill
\includegraphics[width=0.32\textwidth, height=4cm]{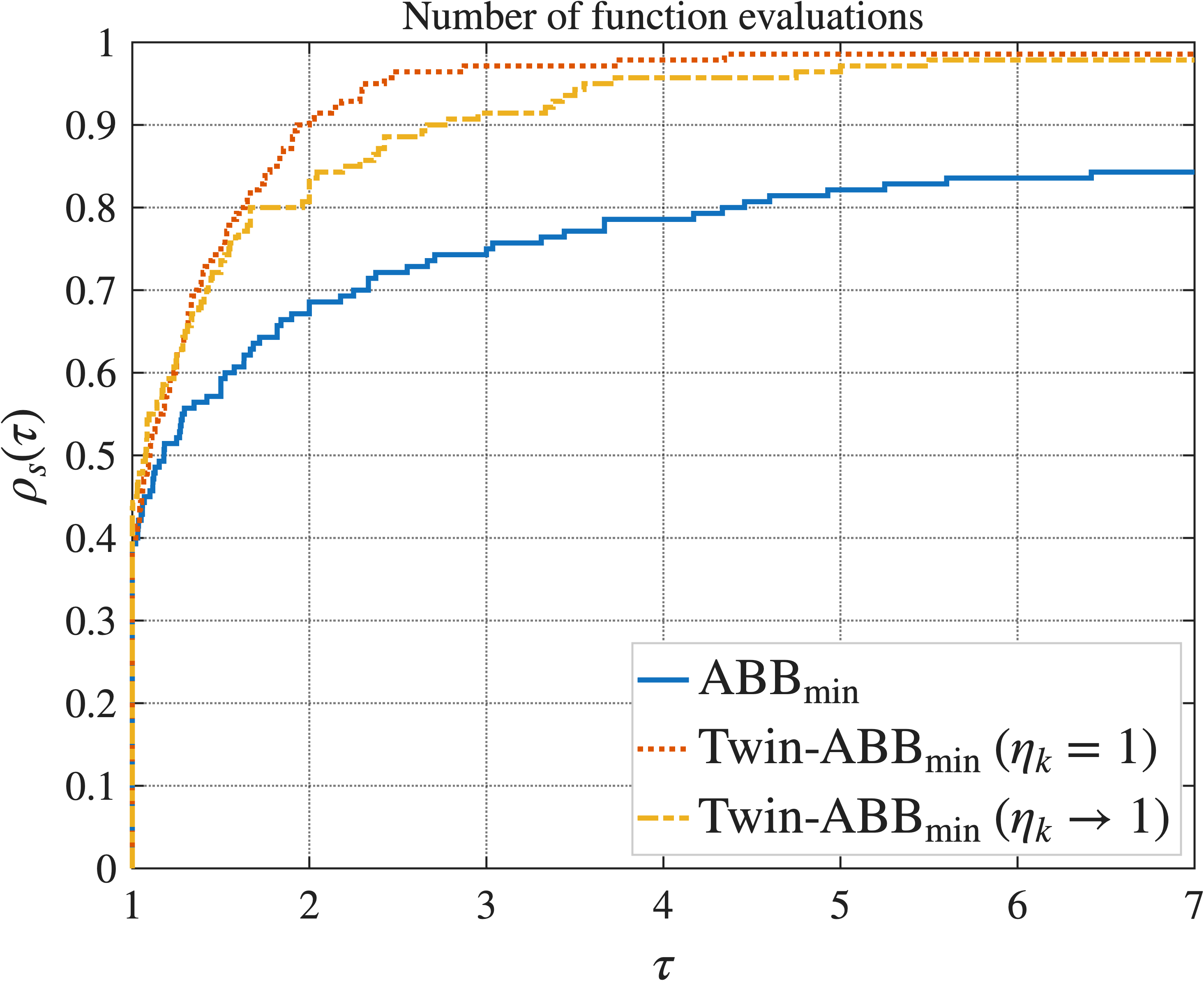}
  \hfill
\includegraphics[width=0.32\textwidth,  height=4cm]{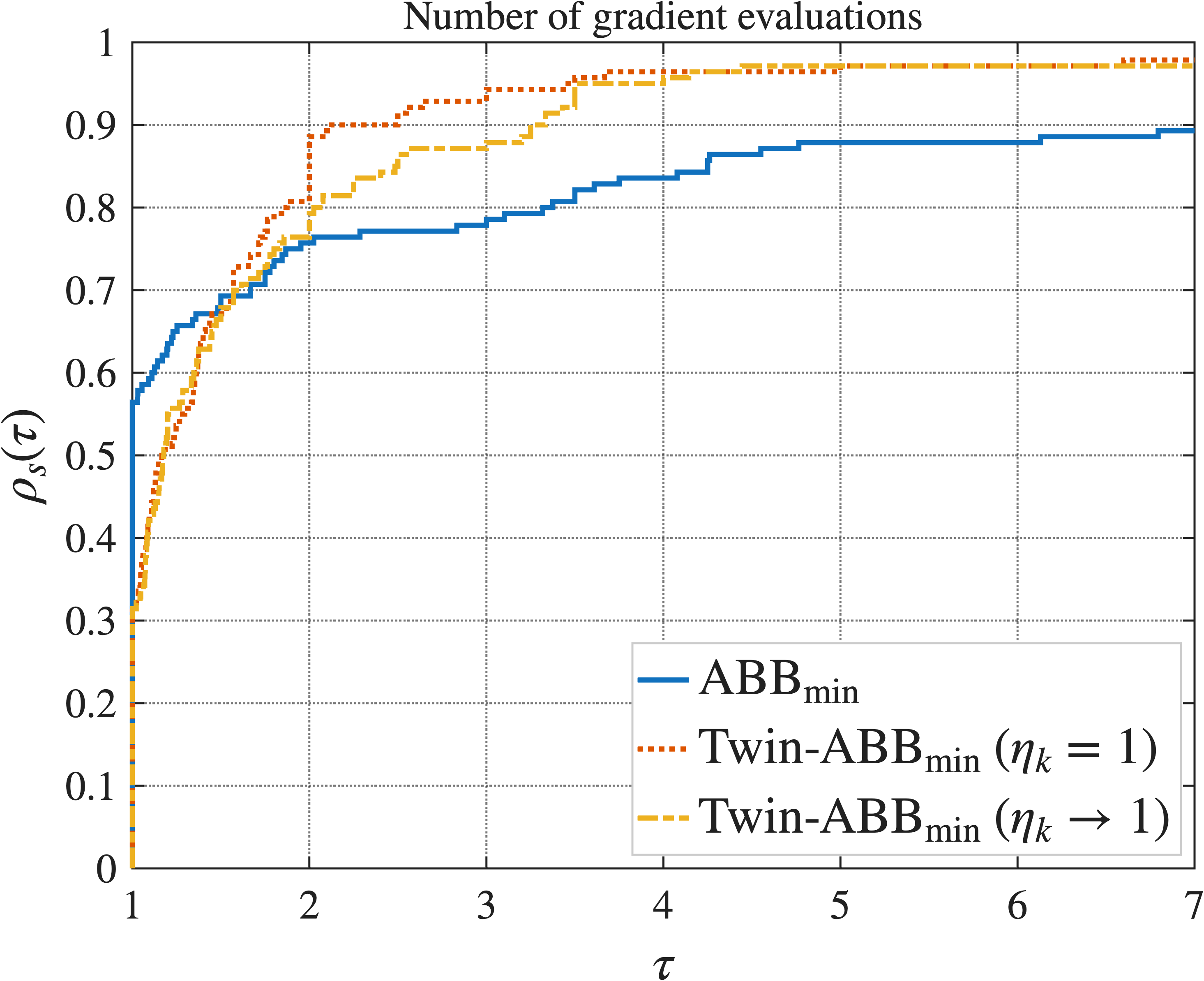} 
\caption{Performance profile on the CUTEst test set.}
\label{fig:perf_cutest_nonquad}
\end{figure}

\begin{figure}[htb!]
\centering
\includegraphics[width=0.4\textwidth,height=4cm]{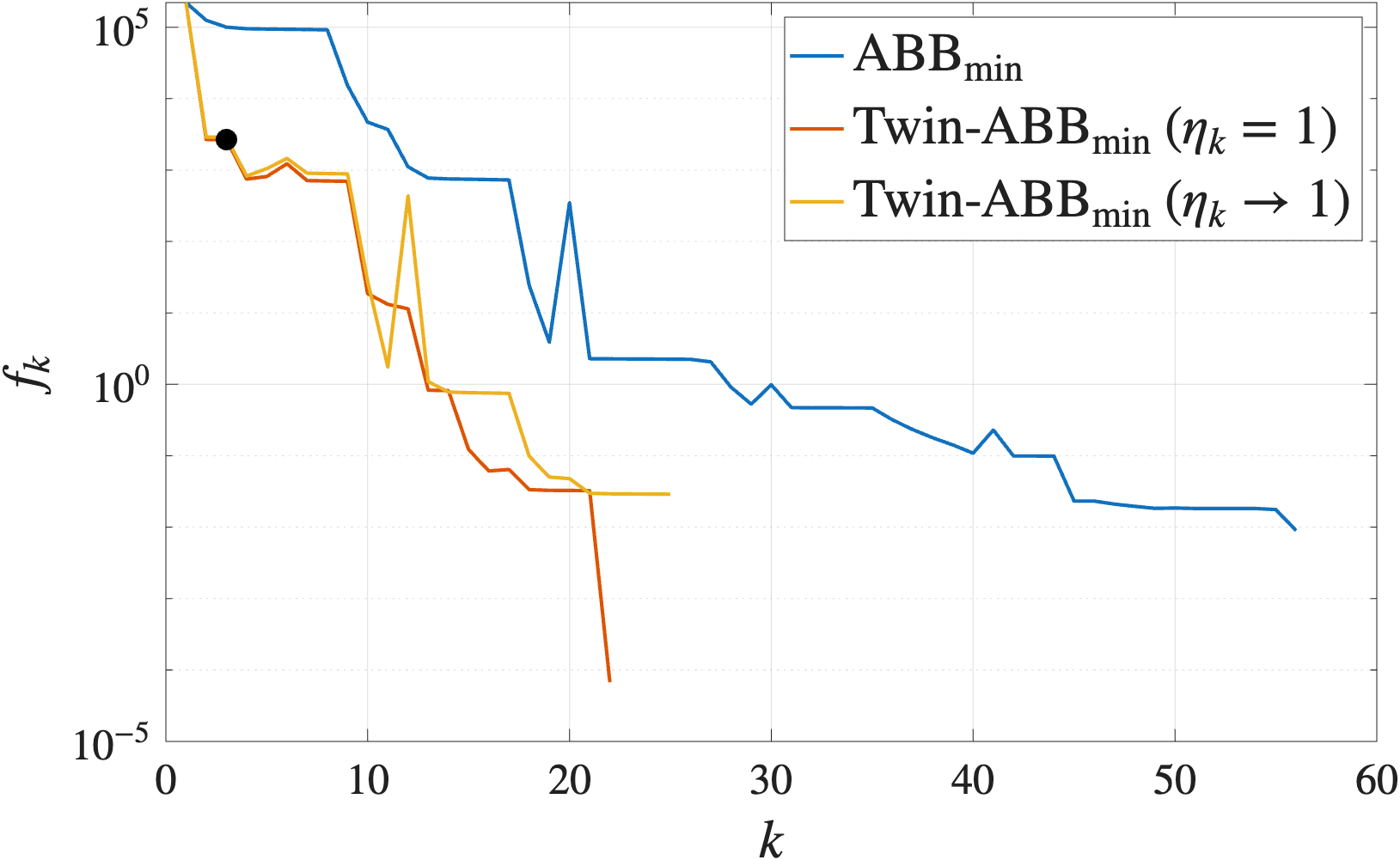}
  \hfill
\includegraphics[width=0.4\textwidth,height=4cm]{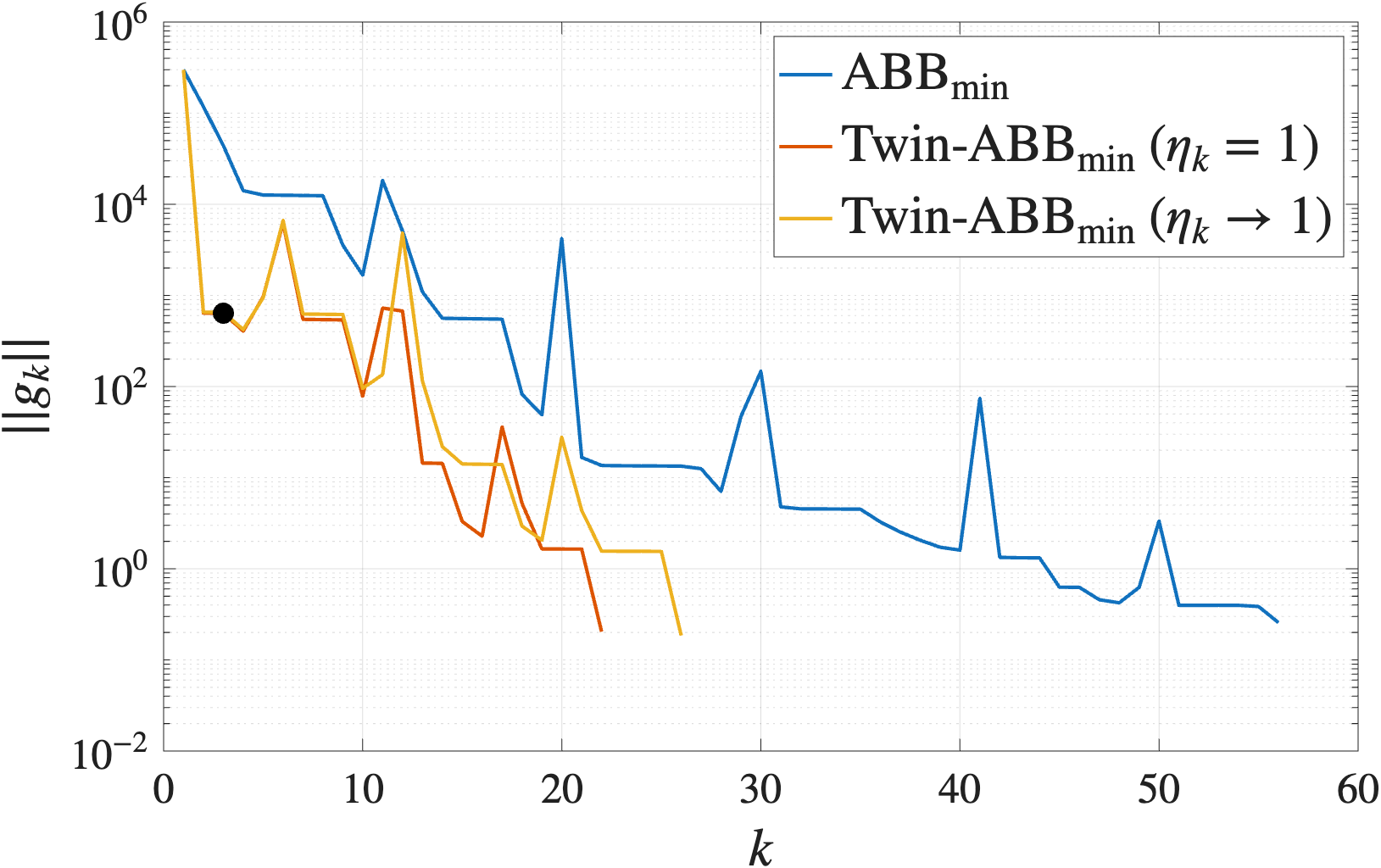}
\caption{Convergence history on the Problem 11 from CUTEst test set. The black dot represents the iterate in which the method switches to ABB$_{\min}$, ($f_k=\min \{f(\bx_k),f(\bz_k)\}), \|g_k\| = \min \{\|g(x_k)\|,\|g(z_k)\|\} $.}
\label{fig:cutest_nonquad_example}
\end{figure}

The computational results are summarized in Figure~\ref{fig:perf_cutest_nonquad}. Performance profiles indicate that the proposed Twin strategy outperforms the standard reference model ABB$_{\min}$ in all metrics evaluated. Although the standard ABB$_{\min}$ model retains a marginal advantage only at $\tau=1$ in terms of the number of iterations and gradient evaluations, it is subsequently outperformed. Our Twin strategy proves to be more efficient from the early stages. As with quadratic problems, the impact of Twin steps lies in their ability to provide initial conditions that are particularly favorable for the ABB$_{\min}$. This is clearly illustrated in the example shown in Figure \ref{fig:cutest_nonquad_example}, which exhibits behavior that is entirely analogous to that observed for the example in Figure \ref{fig:convergence_history}. About the behavior of the stepsizes $\alpha_k$ and $\beta_k$, numerical results show that the unconstrained Twin system \eqref{eq:system2x2_normalized} yield strictly positive stepsizes in more than 99.4\% of total iterations.

\section{Conclusions}
\label{sec:conclusions}
We have introduced a new strategy for gradient-based optimization, inspired by the Twin Kaczmarz method for linear systems \cite{VanLith2021}. The core innovation is to evolve two simultaneous search processes that cooperate through a \textit{Twin-Step} mechanism, selecting stepsizes that minimize the Euclidean distance between the iterates rather than evaluating the objective function in isolation. To establish favorable initial conditions for the algorithm, in line with the theoretical analysis carried out, the auxiliary starting point $\bz_0$ is explicitly constructed in Section~\ref{sec:z0} to foster orthogonality between initial search directions.

From a theoretical perspective, we have provided a convergence analysis, highlighting in particular that the convergence of the mutual distance is governed by the angle between the search directions.
In the early iterations, the Twin method generally results in a significant reduction in the objective function. 
The main risk of the Twin approach is near-parallelism between search directions, leading to big stepsizes and potentially divergence.
This can be solved in two ways. First, we can mitigate this by using damping factors $\eta_k$ (see Section~\ref{sec:nonquadcase}) although a damping designed to keep the directions separated either prevents convergence (Remark~\ref{rem:open_loop}) or, in its closed-loop form, yields only short steps of order $1/L$ (Proposition~\ref{prop:closed_loop} and Remark~\ref{rem:lipschitz_limit}); finding a damping strategy that keeps the directions separated while preserving the long Twin steps is an interesting open question. 
For our experiments however, we have instead chosen a second approach:
in Section~\ref{sec:algo} we have proposed to combine our method with a standard gradient method for the later stages. The Twin-ABB$_{\min}$ algorithm starts with the Twin gradient algorithm in the early optimization phase and adaptively switches to  ABB$_{\min}$ when the Twin coupling becomes collinear. {This design is supported by the theory: the convergence of the Twin iterates is governed by the ratio $R_k = \|\bd_k\|/(1-|\gamma_k|)$, which the switching test keeps under control, so that the method converges both when it remains in the Twin phase (Proposition~\ref{prop:hybrid_twin_phase}) and when it switches to ABB$_{\min}$.}
In our numerical experiments, both in the quadratic case and in the general case, the first stage of the algorithm produces a significant and rapid decrease in the objective function, providing initial conditions that prove particularly favorable for ABB$_{\min}$. We believe that our analysis is still preliminary, although it shows extremely interesting numerical results that seem to warrant further investigation. In particular, this should lead to the development of a more satisfactory theory that can also support the design of efficient algorithms in which the twin strategy plays an even more significant role than in the Twin-ABB$_{\min}$.

\section*{Acknowledgments}
This work has been partially supported by the Italian Ministry of University and Research (MIUR) through the  PRIN 2022 ``Spatio-temporal Functional Marked Point Processes for probabilistic forecasting of earthquake'' CUP B53C24006340006. Part of this work has been carried out while the first author visited the second author at TU Eindhoven.

\footnotesize
\bibliographystyle{plain} 
\bibliography{biblio2}

\end{document}